\documentclass[english,11pt]{amsart}
\usepackage{amssymb,amsmath,amsthm}
\usepackage{srcltx}
\usepackage{babel,graphicx}
\usepackage{color}
\usepackage[latin1]{inputenc}
\newtheorem{theorem}{Theorem}[section]

\newtheorem{corollary}[theorem]{Corollary}
\newtheorem{lemma}[theorem]{Lemma}
\newtheorem{proposition}[theorem]{Proposition}
\theoremstyle{definition}
\newtheorem{definition}[theorem]{Definition}

\theoremstyle{remark}
\newtheorem{remark}[theorem]{Remark}

\numberwithin{equation}{section}

\newcommand{\fin}{\mbox{}\hfill $\square$ \\[0.2cm]}

\newcommand{\Timet}{\textsf{t}}
\newcommand{\Times}{\textsf{s}}
\newcommand{\Timeu}{\textsf{u}}
\newcommand{\Timeq}{\textsf{q}}

\newcommand{\N}{\mathbf{N}}
\newcommand{\Z}{\mathbf{Z}}
\newcommand{\Q}{\mathbf{Q}}
\newcommand{\R}{\mathbf{R}}
\newcommand{\C}{\mathbf{C}}
\newcommand{\D}{\mathbf{D}}

\newcommand{\FF}{\mathcal{F}}
\newcommand{\LL}{L}

\newcommand{\oo}{\gamma}
\newcommand{\OO}{\Gamma}

\begin{document}

\title[Dimension of harmonic currents]{The dimension of harmonic currents on foliated complex surfaces}
\author{Bertrand Deroin, Christophe Dupont \& Victor Kleptsyn}%

\address{CNRS - Laboratoire AGM - UMR CNRS 8088 - Universit\'e de Cergy-Pontoise}
\email{bertrand.deroin@cyu.fr}

\address{IRMAR - UMR CNRS 6625 - Universit\'e de Rennes}
\email{christophe.dupont@univ-rennes.fr}

\address{CNRS - IRMAR - UMR CNRS 6625 - Universit\'e de Rennes}
\email{victor.kleptsyn@univ-rennes.fr}

\thanks{The authors benefited from the support of CY Initiative (ANR-16-IDEX-0008) for B.D.,  of ERC (GOAT 101053021) for C.D., of Labex Lebesgue (ANR-11-LABX-0020-01) for C.D. and V.K., and of ANR Gromeov (ANR-19-CE40-0007) for V.K. We also thank the Research in Paris program of Institut Henri Poincar\'e for bringing us together and hosting us during fall 2021.}





\begin{abstract}
Let $\FF$ be a singular holomorphic foliation on an algebraic complex surface $S$, with hyperbolic singularities and no foliated cycle.   We prove a formula for the transverse Hausdorff dimension of the unique  harmonic current, involving the Furstenberg entropy and the Lyapunov exponent. In particular, we extend Brunella's inequality to every holomorphic foliation $\FF$ on $\mathbb P^2$: if $\FF$ has degree \(d \geq 2 \), then the Hausdorff dimension of its harmonic current is smaller than or equal to ${d-1 \over d+2}$, in particular the harmonic current is singular with respect to the Lebesgue measure. We also show that the Hausdorff dimension  of the harmonic current of the Jouanolou foliation of degree $2$ is equal to \(1/4\), and that the same property holds for topologically conjugate foliations on $\mathbb P^2$.
\end{abstract}

\maketitle
\tableofcontents 

\part{Overview, statements and notations}

{\bf Overview. } Let \(\mathcal F\) be a holomorphic foliation on a compact K\"ahler surface \(S\), with hyperbolic singularities and no foliated cycle. For instance, for any \(d\geq 2\), a holomorphic foliation of the complex projective plane $\mathbb P^2$ has this property if it is chosen in a real Zariski dense open subset. For a foliation \(\mathcal F\) as before, leaves are transcendental, and by Ahlfors' theory, there is no entire curve tangent to the foliation. In particular, any leaf \(L\) can be uniformized by a covering map \( \phi : \mathbb D \rightarrow L\) and equipped with the Poincar\'e metric \(g_P\) of constant gaussian curvature\footnote{hence  \(\phi ^* g_P = 2 \frac{|dz|}{1-|z|^2}\) on the unit disc $\mathbb D$.} $-1$.  This produces a continuous metric on the tangent bundle $T_\FF$ with a logarithmic pole at singular points. 

There is a remarkable unique ergodicity property in this context: normalized leafwise hyperbolic balls of large radius converge to a unique current (independent of the leaf) after a Cesaro type mean\footnote{The sequence of \((1,1)\)-currents $T_ r := c_r \phi_* \log ^+ \frac{r}{|z|}$ converges when \(r\) tends to \(1\) to a  current \( T\)  which does not depend on the leaf \(L\) nor on its uniformization \(\phi\).}. This was proved by Fornaess-Sibony  on $\mathbb P^2$  \cite{FS} and by Dinh-Ngyuen-Sibony  on any K\"ahler complex surface \cite{DNS_unique ergodicity} (see  \cite{Deroin-Kleptsyn} for a similar property involving leafwise paths). The current $T$ is positive, $\FF$-directed and \(dd^c\)-closed, it is called the \emph{harmonic current} of $\FF$.  
An important consequence of the unique ergodicity theorem is that \(\mathcal F\) carries a unique pseudo-minimal set\footnote{A pseudo-minimal set is a closed subset of $S$, saturated by the leaves of $\FF$, not included in the singular set, in which every leaf is dense.}  and that the current $T$ is supported on this set. 

 When $S$ is an algebraic surface, Nguyen \cite{Nguyen} proved that the positive measure $\mu := T\wedge \text{vol}_{g_P}$  has finite mass\footnote{In the remainder, $T$ is normalized such that $\mu$ has mass one.} 
 on $S$. He precisely established the $\mu$-integrability of functions $\rho$ with logarithmic growth at singular points. This integrability property will be fundamental in the present work. 
 
In local coordinates $(z,t) \in \mathbb D \times \mathbb D$, foliated by horizontal discs, the evaluation of $T$ on a $(1,1)$-form $\omega$ takes the form  
\[ T (\omega ) = \int_{\mathbb D}  \left( \int_{\mathbb D\times \{t\}}  \tau (\cdot , t)\, \omega \right) \nu(dt d\overline{t} ) \]  
where \(\nu \) is a Radon measure on the vertical disc and \( \tau\in L^1 (dz d\overline{z} \nu (dtd\overline{t} ) ) \) restricts to \(\nu \)-a.e. plaque \( \mathbb D \times \{t\}\) as a non negative harmonic function of \(z\). One may think to  the current \(T\) as a family of measures on holomorphic transversals $\Sigma$, varying harmonically. Precisely, if $\Sigma$ is given by $\{ (z,t), z= z(t)\}$,  the transverse measure is defined by
\begin{equation} \label{eq: local expression harmonic current2} T_{\vert \Sigma} := \tau (z(t), t) \nu (dtd\overline{t}).\end{equation}

The goal of the present article is to provide a formula for the Hausdorff dimension of those measures $T_{\vert \Sigma}$, in the vein of Ledrappier-Young theory. We will prove that, when not identically vanishing, all the $T_{\vert \Sigma}$'s have the same dimension, equal to the \emph{quotient of an entropy by a Lyapunov exponent}, as usual in unidimensional complex dynamics. Moreover, we prove that the $T_{\vert \Sigma}$'s are \emph{exact dimensional}. Our result applies for foliations on $\mathbb P^2$: we obtain an \emph{upper estimate for the dimension} involving the degree of the foliation, and obtain an \emph{exact value} for the Jouanolou foliation.

\vspace{0.2cm} 

{\bf Entropies and Lyapunov exponent.} 
The \textit{dynamical (Furstenberg) entropy} of the harmonic current $T$ is the non negative number defined by
\[ h_D := \int _S dd^c_{\mathcal F} \log \tau \wedge T =  - \int_S \Delta_{g_P} \log \tau \ d\mu  . \] 
It was introduced by Kaimanovich \cite{Kaimanovich} and Frankel \cite{Frankel}: it is the counterpart in the context of foliations of the Furstenberg entropy of harmonic measures related to random matrix products. It vanishes if and only if $\FF$ supports an invariant measure\footnote{which means that is \(T\) is closed, or is a foliation cycle in Ruelle-Sullivan's terminology.}. 
 The \textit{leaf entropy} is defined by  
\begin{equation*}  h_L : = \lim _{\Timet\rightarrow +\infty} - \frac{1}{\Timet} \int_{L_x}  p(x,y,\Timet) \log p(x,y,\Timet) \mathrm{vol}_{g_P} (dy) , \end{equation*}
where \(p(x,y,t)\) is the leafwise heat kernel associated to \(g_P\), the limit does not depend on the $\mu$-generic point \(x\), see \cite{Kaimanovich, L3}. Observe that
\[ 0\leq h_D \leq h_L \leq 1.\]
The inequality $h_D \leq h_L $ is due to Kaimanovich \cite{Kaimanovich}, we provide a proof in an Appendix for reader's convenience. The bound $h_L \leq 1$ comes from the fact that the entropy of the hyperbolic disc is equal to \(1\), and that the entropy increases by performing coverings. The \textit{Lyapunov exponent} is 
\begin{equation}\label{eq: cohomological expression Lyapunov} 
\lambda :=  \frac{T\cdot N_{\mathcal F} }{T \cdot T_{\mathcal F} } , 
\end{equation}
where\footnote{ \( T\cdot E\) denotes the evaluation of \(T\) on the curvature of an hermitian metric on \(E\), for every line bundle \(E\) over the complex surface $S$.} \( N_{\mathcal F}\) and \(T_{\mathcal F}\) denote the normal and tangent bundle of $\FF$.
Under our assumption (hyperbolic singularities and no foliated cycles for $\FF$), the Lyapunov exponent is negative, see the Notes at the end of this Section.\\

{\bf Results.} In the present work we prove that $h_D$ is  the opposite of the exponential rate of the Radon-Nikodym derivative (wrt $T_{\vert \Sigma}$) of the holonomy maps \( h_{\gamma_{[0,t]}} \) for generic (wrt a measure $W^\mu_{g_P}$) leafwise continuous path $\gamma$. Likewise,  $\lambda$ will be the exponential rate of the derivative of generic holonomy maps  \( h_{\gamma_{[0,t]}}\), justifying the terminology of Lyapunov exponent. 

\begin{theorem}\label{thmB}
Let $\FF$ be a foliation on an algebraic surface $S$ with hyperbolic singularities and no foliated cycle and $T$ be its harmonic current. For any holomorphic transversal $\Sigma$ intersecting the support of $T$, we have for the transversal measure $T_{\vert \Sigma}$ defined in Equation \eqref{eq: local expression harmonic current2}:  
$$ \textrm{for } T_{\vert \Sigma}-a.e. \ t \in \D \ , \  \lim_{r \to 0} {\log T_{\vert \Sigma} (\D_t(r)  )\over \log r} = {h_D \over \vert \lambda \vert},$$
In particular, the Hausdorff dimension\footnote{The Hausdorff dimension of a measure $\nu$ on the disc $\D$ is the infimum of the Hausdorff dimension of Borel sets $E \subset \D$ satisfying $\nu(E) = \nu(\D)$ \cite{Y82}.}   of $T_{\vert \Sigma}$ is equal to  ${h_D /  \vert \lambda \vert}$.
\end{theorem}

Let us recall that for every holomorphic foliation $\FF$ of degree \(d \) on $\mathbb P^2$, we have $N_{\mathcal F} \simeq \mathcal O(d+2)$ and $K_{\mathcal F} \simeq \mathcal O(d-1)$. We deduce from Theorem \ref{thmB} and Equation \eqref{eq: cohomological expression Lyapunov} the following Corollary, it extends Brunella's upper estimate \cite{Brunella conforme} established for the special case of \emph{conformal} harmonic currents. 

\begin{corollary}
Let $\FF$ be a holomorphic foliation of degree \(d\geq 2\) on $\mathbb P^2$ with hyperbolic singularities and no foliated cycle. Then the Hausdorff dimension ${h_D \over  \vert \lambda \vert}$ of its harmonic current $T$ is  \( \leq \frac{d-1}{d+2}\). In particular, $T$ is never absolutely continuous with respect to Lebesgue measure. 
\end{corollary}

From Theorem \ref{thmB}, to obtain an explicit value for the Hausdorff dimension of $T$, we need to compute $h_D$, which is generally a difficult problem. For that purpose we prove the following result, generalizing works by Kaimanovich \cite{K2000} and Ledrappier \cite{L1, L2} for random matrix products. 

\begin{theorem}\label{t: dischh}  Let $\FF$ be a foliation on an algebraic surface $S$ with hyperbolic singularities and no foliated cycle.  If the holonomy pseudogroup of $\FF$ in restriction to the pseudo-minimal set is discrete, then $h_D = h_L$. In particular, $T_{\vert \Sigma}$ is exact dimensional, with Hausdorff dimension equal to ${h_L \over  \vert \lambda \vert }$.
\end{theorem}

Observe that $h_D$ is a conformal invariant. It is unclear wether $h_L$ is also a conformal invariant, but Theorem \ref{t: dischh} establishes this fact when the holonomy pseudo-group is discrete in restriction to the pseudo-minimal set. The idea of the proof is to observe that the growth rate of the cardinality of leafwise separated subsets is at least $h_L$, and that each separated position is weighted at the rate $-h_D$. Some arguments are borrowed from \cite{Deroin-Dupont}. Difficulties appear here due to the singular points (see for instance Proposition \ref{p: adapt}). The fact that $\lambda$ is negative implies that the holonomy maps are contracting and allows to use the discreteness assumption.\\

As a classical fact, if the holonomy pseudogroup of $\FF$ is not discrete, then the support of $T$ is an analytic submanifold of $S$. For foliations in the stability component of the Jouanolou foliation\footnote{The Jouanolou foliation of degree \(d\) on $\mathbb P^2$ is defined by the vector field \( y^d \partial_x + z^d \partial _y + x^d \partial _z\). It is conjectured in \cite{AlvarezDeroin} (proved for \(d=2\)) that those foliations are structurally stable, with a non trivial fractal Julia set, and simply connected leaves appart from a countable number. This would imply that  the  dimension of its harmonic current is  \( \frac{d-1}{d+2}\).} of $\mathbb P^2$ of degree $2$, we are able to discard this analytic regularity, and also to prove that $h_L = 1$.

\begin{corollary} \label{c: dimension Jouanolou}
Let $\FF$ be foliation on $\mathbb P^2$ of degree \(2\) topologically conjugated to  the Jouanolou foliation. Then $\nu$ is exact dimensional, with Hausdorff dimension equal to \(1/4\). \end{corollary}

The proof precisely needs results concerning the Jouanolou foliation, with among others the fact that it does not have any transcendental leaf (Jouanolou) neither any minimal set (Camacho-de Figuereido), the fact that its Fatou set is not empty and that it has the Anosov property (Alvarez-Deroin).
 
\vspace{0.2cm}

{\bf Outline of the proof of Theorem \ref{thmB}.} Let $\Gamma$ denote the set of leafwise continuous paths parametrized by $\R^+$ and let $W^\mu_{g_P}$ denote the measure on $\Gamma$ given by the $\mu$-average of the Wiener measures induced by the leafwise Poincar\'e metric $g_P$. 
We get a dynamical system $(\Gamma, \sigma, W^\mu_{g_P})$ where $\sigma$ is a shift semi-group acting on $\Gamma$ (see Section \ref{ss: harmonicm2} for details). 

\vspace{0.2cm} 

{\bf First step.} We introduce adapted metrics with respect to the singular points. Near those points, our smooth metrics $m$ on $N_\FF$ and $g_s$ on $T_\FF$ are respectively defined by the modulus of a holomorphic $1$-form defining $\FF$ and by assigning unit norm to a vector field defining $\FF$, see Section \ref{s: dgf}. In particular $m$ and $g_s$ have zero curvature near the singular points. 

\vspace{0.2cm} 

{\bf Second step.}
We prove that the $g_s$-length of leafwise continuous paths is $W^\mu_{g_P}$-integrable, see Proposition \ref{p: integrability}, this will be crucial. We use Nguyen's $\mu$-integrability of  functions $\rho$ with logarithmic growth at singular points. 

\vspace{0.2cm} 

{\bf Third step.} We prove a general Limit Theorem for measurable families of leafwise $1$-forms, see Theorem \ref{p: expression of a}.  Thanks to the $W^\mu_{g_P}$-integrability of the $g_s$-length of leafwise continuous paths, the assumptions of the Limit Theorem only require $g_s$-boundedness. The arguments involves Kingman's ergodic theorem and Dynkin's formula. A first application is Proposition \ref{d: lyapunov exponent}, which gives 
the existence of $\lambda  \in \R$ such that for $W_{g_P}^\mu$-a.e. $\gamma \in \OO$,    
\begin{equation}\label{lambdapointw}
 \lim _{\Timet\rightarrow +\infty} {1 \over \Timet} \log \vert Dh_{\oo_{[0,\Timet]}, \Sigma_0, \Sigma_\Timet}(\oo(0)) \vert_{m}  = \lambda  .
 \end{equation}
 A second application is Proposition \ref{p: integral of beta} which  yields for $W_{g_P}^\mu$-a.e. $\gamma \in \OO$,  
\begin{equation}\label{hpointw}    \lim _{\Timet\rightarrow +\infty} \frac{1}{\Timet} \log \frac{D (h_{\gamma_{[0,\Timet]}, \Sigma_0, \Sigma_\Timet})_*^{-1} T_{|\Sigma_\Timet} }{D T_{|\Sigma_0}} (\gamma(0)) = - h_D . 
 \end{equation} 
 To prove Theorem \ref{thmB}, we need to extend those pointwise limits to nearby leaves, which is the aim of the next two important steps.

\vspace{0.2cm} 
 
{\bf Fourth step.} Given a $W_{g_P}^\mu$-generic leafwise path $\gamma$, we show in Proposition \ref{prop:holodist} that there exists a disc $\D(\delta_\gamma) \subset \Sigma_{\gamma(0)}$ such that  the $\gamma$-holonomy mapping $h_{\gamma_{[0,t]}}$ is well defined on $\D(\delta_\gamma)$ for every $t \geq 0$ and is contracting with the exponential rate $\lambda$. For that purpose we introduce in Section \ref{s: ccd} a finite family of foliation boxes $(U_k)_{k \in K}$ such that, near every singular point, the boxes are built on the product structure provided by linearization coordinates. Then we prove that the restricted path $\gamma_{[n , n+1]}$ crosses at most $n\epsilon$ boxes (up to homotopy): the number $Q$ of crossings is indeed bounded below by the $g_s$-length of $\gamma_{[n , n+1]}$  (Proposition \ref{p: estimq}), which is $W^\mu_{g_P}$-integrable. Proposition  \ref{prop:holodist} then follows by induction, concatening the successive $n\epsilon$ restrictions $\gamma_{[n , n+1]}$, using $\vert Dh_{\oo_{[0,n]}}(\oo(0)) \vert_{m} \simeq e^{n (\lambda \pm \epsilon)}$ (given by Equation \eqref{lambdapointw}) and Koebe distortion Theorem.

\vspace{0.2cm} 

{\bf Fifth step.} 
We extend the pointwise asymptotic derivative $$ { D (h_{\gamma_{[0,n]}})_*^{-1} T_{|\Sigma_n} \over  D T_{|\Sigma_0} }(\gamma(0)) \simeq e^{-n (h_D \pm \epsilon)}$$ provided by Equation \eqref{hpointw} to the disc $\D(\delta_\gamma)$.
Our aim is to get 
\begin{equation}\label{eq: entr}
{ T_{\vert \Sigma_n}( h_{\oo_{[0,n]}} (\D(\delta_\gamma)) ) \over T_{\vert \Sigma_0} (\D(\delta_\gamma))} \simeq e^{-n (h_D \pm 2\epsilon)}.
\end{equation}
As in the fourth step, we proceed by induction on the successive  restrictions $\gamma_{[n , n+1]}$. We precisely study the multiplicative error $e^{\eta_n}$ between 
$${ T_{\vert \Sigma_{n+1}}( h_{\oo_{[0,n+1]}} (\D(\delta_\gamma))) \over T_{\vert \Sigma_n} (h_{\oo_{[0,n]}} (\D(\delta_\gamma)))} \textrm{ and  } {D (h_{\gamma_{[n,n+1]}})_*^{-1} T_{|\Sigma_{n+1}} \over D T_{|\Sigma_n}} (\gamma(n)) . $$
By classical geometric measure theory,  $\eta_n$ tends to $0$ since $h_{\oo_{[0,n]}}$ is contracting. Moreover, thanks to Harnack inequality for positive harmonic functions, $\eta_n$ is bounded above by the $W^\mu_{g_P}$-integrable number $Q$ of crossings, introduced in the fourth step. Birkhoff ergodic theorem then implies that the full error term $\eta_1 + \ldots + \eta_n$  is $ \leq n \epsilon$. This yields the desired Equation \eqref{eq: entr}.  

\vspace{0.2cm} 

{\bf Sixth step.} Combining the third and fourth step, we get 
\begin{equation*}
{ T_{\vert \Sigma_n} (\D_{\gamma(n)}(\delta_\gamma e^{n(\lambda \pm \epsilon) } ))  \over T_{\vert \Sigma_0} (\D_{\gamma(0)}(\delta_\gamma))} \simeq { T_{\vert \Sigma_n}( h_{\oo_{[0,n]}} (\D(\delta_\gamma)) ) \over T_{\vert \Sigma_0} (\D(\delta_\gamma))} \simeq e^{-n (h_D \pm 2\epsilon)}.
\end{equation*}
To complete the proof of Theorem \ref{thmB}, we have to replace the target point $\gamma(n)$ by a fixed $\nu$-generic point $t$ in a  transversal $\Sigma$. We manage this by working in the natural extension of the dynamical system $(\Gamma, \sigma, W^\mu_{g_P})$.

\vspace{0.2cm} 

{\bf Notes and related results.}  

For random matrix products on $SL_2(\R)$ and $SL_2(\C)$, the dimension of harmonic measures was computed in \cite{L00, L1}: ${\log \nu (\D_t(r)  ) /  \log r}$ converges in probability to  Furstenberg entropy/Lyapunov exponent. 

For Ricatti foliations on ruled surfaces, the dimension of harmonic currents was investigated in \cite{Deroin-Dujardin} in the spirit of \cite{L00, L1}, here the limit of  ${\log \nu (\D_t(r)  ) /  \log r}$ is bounded above (in probability) by $h_D / \vert \lambda \vert$. 

For random matrix products on $SL_2(\R)$, the exactness of the dimension of harmonic measures was proved in \cite{HS} and extended to $SL_d(\R)$ in~\cite{LL}.

The notion of discreteness for the holonomy pseudogroup of a foliation was introduced in \cite{Ghys93, N94}. It is proved  in \cite{Deroin-Dupont} that the transverse Hausdorff dimension of a non singular minimal set with discrete holonomy pseudogroup is bounded below by $h_L / \vert \lambda\vert$ (without introducing $h_D$). This result is now a consequence of Theorem \ref{t: dischh}.  
 
 The cohomological formula \eqref{eq: cohomological expression Lyapunov} for the Lyapunov exponent was proved in the non singular context in \cite{Deroin} and in the singular one in \cite{NguyenII}. Along the present article, we provide an alternative proof in the singular context. The difference is that we use the smooth metric $m$ on $N_\FF$, instead of an ambiant metric on $S$, which makes easier computations. To be more precise, a first step is Proposition \ref{p: cohomological interpretation of the mass of the harmonic measure},  asserting that ${T\cdot T_{\FF}} = - {1 \over 2 \pi} \int_S T\wedge \text{vol}_{g_P}$. This is done by comparing the curvature  of the leafwise Poincar\'e metric $g_P$ with the curvature of the metric $g_s$ on $T_\FF$. For that purpose we use estimates for  the ratio $e^{-\xi}$ of these two metrics \cite{DNS_Entropy II} and the fact that $T\wedge \text{vol}_{g_P}$ has finite mass on $S$ \cite{Nguyen}. The second step is to establish that $\lambda =  - 2\pi \, T\cdot N_{\FF}$, which follows from the Limit Theorem.
 
The fact that $\lambda$ is negative is a consequence of the property \( T\cdot N_{\mathcal F} >0 \). This was proved in \cite{Deroin-Kleptsyn} in the non singular context and in \cite{NguyenII} in the singular one, see also  \cite{DDK} for another proof. The fact that $\lambda$ is the exponential rate of the holonomy derivative was established in \cite{Deroin, Deroin-Kleptsyn} in the non singular context,  in \cite{Nguyen} in the singular context. 

Concerning our Limit Theorem, contrary to the method used in \cite{Deroin-Kleptsyn}, we shall not need the transverse regularity of the Poincar\'e metric, here this is provided by the uniqueness of the harmonic current $T$. Analogous Limit Theorems were already established in other situations \cite{Candel, Deroin, NguyenAMS}. 

\vspace{0.2cm} 
 
{\bf Notations.} 

\small

\vspace{2mm}

\begin{tabular}{ll}
\(S\) & complex algebraic surface 
\\
\(\FF\) & foliation on \(S\) with hyperbolic singularities and no foliated cycle
\\
\( \text{sing}(\FF) \) & singular set of \(\FF\) 
\\
\( S^* \) &  \( S \setminus \text{sing}(\FF) \), the regular part of \(\FF\) 
\\
\( \mathcal L \), \( \mathcal L_x \) & \( \mathcal L \)  is a leaf of $\FF$, \( \mathcal L_x \)  is the leaf containing $x \in S^*$.  
\\
\( \rho \) &\(  - \log r \), where \(  r \) is the distance in $S$ to the singular set
\\
$A$ & angular domain in $\C$ defined by a hyperbolic singularity
\\
$d_e^A$, $d_P^A$ & euclidian and Poincar\'e metric in the angular domain $A$ 
\\
\( N_{\FF} \) & normal bundle to \( \FF\)
\\
\( m\) & hermitian metric on \( N_{\FF}\)  
\\
\( T_\FF \), \( K_{\FF}  \) & tangent bundle to \(\FF\), canonical bundle to \( \FF \)
\\
\( g_s\) & smooth hermitian metric on \( T_\FF\) 
\\
\( g_P\) & Poincar\'e metric on \( T_\FF \) 
\\
\( \text{vol}_{g} \) & leafwise volume \((1,1)\)-form associated to a metric \(g\) on \( T_\FF \)
\\
\( \Delta_g\) & leafwise Laplacian operator associated to a metric \(g\) on \( T_\FF \)
\\
 \( T \) & harmonic current
 \\
 \( \Sigma , t \) & transversal for $\FF$, transversal coordinate or local first integral
\\
 \( \tau (z,t) , \nu \) & in a foliation box \((z,t)\) one has \( T = \tau (z,t) \, \nu (dt) \) 
 \\
 \( \mu \) &  \( T \wedge \text{vol}_{g_{P}} \) 
\\
\(\OO, \gamma \) & $\Gamma$ is the set of leafwise continuous paths \(\oo : [0,+\infty) \rightarrow S^*\) 
\\
\( h_{\gamma_{[a,b]},t,t'} \) & germ of holonomy map from $(\C, t(\gamma(a)))$ to $(\C, t'(\gamma(b)))$
\\
\(\sigma\) & shift semi-group acting on \(\OO\) 
\\
\(\OO ^x\) & subset of \( \OO \) of  leafwise continuous paths $\gamma$ starting at  \(\oo(0)= x \).  
\\
\(\Timet, \Times \) & time parameters of leafwise continuous paths \(\oo : [0,+\infty) \rightarrow S^*\) 
\\
\( W^x_g \) & Wiener measure on \(\OO ^x\) generated by the Laplacian of the metric~\(g\). 
\\
 \( d_{\FF}, d^c_{\FF}, \partial _{\FF},  \overline{\partial} _{\FF} \) & differential operators along the leaves of \(\FF\).  
\end{tabular}

\vspace{2mm}

We have $d^c = \frac{1}{2i\pi} (\partial - \overline{\partial} ) = \frac{1}{2\pi} (\partial_x dy - \partial_y dx)$,   
and \( d d^c = \frac{i}{\pi} \partial \overline{\partial} \), see \cite[\S3, p.144]{Demailly}. On a Riemann surface equipped with a Hermitian metric~\(g\), $dd^c f = \frac{1}{2\pi} (\Delta_g f) \cdot \text{vol} _g$.

\normalsize

\part{Differential geometry and Harmonic currents}

\section{Differential geometry for foliations}\label{s: dgf}

\subsection{Holomorphic foliations on surfaces} \label{clanotion}

A \emph{holomorphic foliation} \(\FF\) on an algebraic complex surface \(S\) is the data of a line bundle \( T_\FF \rightarrow S\) and a morphism \(\chi: T_\FF \rightarrow TS\) which vanishes only at a finite number of points. Such a point is called a singularity of \( \FF\), their set is denoted by \(\text{sing}(\FF)\), and we set \(S^*:= S\setminus \text{sing}(\FF) \).   \(T_\FF\) is called the \emph{tangent bundle}, its dual \( K_{\FF} := T^* \FF\) is called the \emph{canonical bundle}. The vector fields on \(S\) that are images of local sections of \(T_\FF\) by the morphism \(\chi\) form a subsheaf \(\mathcal O (\FF)\) of \(\mathcal O (TS)\) called the \emph{tangent sheaf}.

The \emph{leaves} of $\FF$ are the equivalence classes of the relation: two points of $S^*$ are equivalent if they belong to the same integral curve of locally defined vector fields belonging to \(\mathcal O (\FF)\). A \emph{local first integral} of \(\FF\) is a complex valued function defined on an open subset of \(S\), which is constant along the leaves. On a neighborhood \(U\) of any regular point \(p\) is defined a holomorphic  \emph{foliation chart} \( (z, t) : U \rightarrow {\bf D}\times {\bf D}\) that maps \(\mathcal O (\FF)\) to the sheaf of horizontal vector fields. The function \( t: U \rightarrow {\bf D}\) is a submersion and a local first integral. Every first integral in \(U\) is a function of \(t\). 

\subsection{Differential operators, curvature}\label{sub:DOC}

Let \( d_{\FF}, \partial _{\FF} \) and \( \overline{\partial} _{\FF}\) denote the differential operators along the leaves, we shortly denote $\partial  \overline{\partial} _{\FF} := \partial _{\FF}  \overline{\partial} _{\FF}\).  Setting \( d^c_{\FF} := \frac{1}{2i\pi} (\partial_{\FF} - \overline{\partial}_{\FF} ) \), we get \( d d^c_{\FF} = \frac{i}{\pi} \partial \overline{\partial}_{\FF} \). If $g$ is a hermitian metric on $T_\FF$ and if $\Delta_g$ denotes its leafwise laplacian, then $$  \Delta_g  \cdot \text{vol} _g =  2i \,  \partial \overline{\partial}_{\FF} = 2 \pi \, dd^c_{\FF}  .$$
If $h$ is an hermitian metric on a line bundle $E$ over $S$ (for instance \( T_\FF\), $K_\FF$ or $N_\FF$ introduced below), the curvature form of $h$ is locally defined by 
$$ \Theta_{h} =  - \frac{i}{\pi}\partial \overline{\partial} _{\FF} \log h (s) = -dd^c_\FF \log h (s) , $$ where \( s \) is any non vanishing holomorphic section of $E$.

\subsection{Hyperbolic singularities, angular domain}\label{sub:singpoints} 

Let $p$ be a singular point of $\FF$. There is a germ of vector field \(V\) belonging to the tangent sheaf of the foliation, that vanishes only at \(p\). By Hartog's lemma, such a vector field is the \(\chi\)-image of a generating section of \(T_\FF\). Any other germ of vector field with these properties differ from \(V\) by multiplication by a non vanishing holomorphic function. 

The singular point \(p\) is \emph{hyperbolic} if the two eigenvalues of \(V\) are not \({\bf R}\)-colinear. In that case, by Poincar\'e linearization theorem \cite[Theorem~5.5]{IY}, there exist a neighborhood \(B_p\) of \(p\) and coordinates \((x,y): B_p \rightarrow \D\times \D \) in which  \(V\) takes the form 
\begin{equation}\label{eq: V}
 V= a x \partial _x + b y \partial _y . 
 \end{equation}
 We denote $B_p^* := B_p \setminus \{ p \}$. With no loss of generality (multiplying \(V\) by a non zero complex number), we assume that $a,b$ have positive real part. We also assume that \( (x,y)\) is defined on a  domain containing the closure of~\(B_p\).  In these coordinates, the leaves of \( \FF \) consist in the two separatrices (the two axis punctured at $p$) and in the image of the complex curves 
\begin{equation}\label{eq: parametrization}  \Gamma_{(x_0,y_0)} :  u\mapsto (x_0 \exp (a u) , y_0 \exp (b u)) , \end{equation} 
 where $(x_0 , y_0) \in \mathbb S^1 \times \mathbb S^1$. These leaves accumulate on both separatrices and are not closed, even if we take out $p$. If the leaf is not a separatrice, then $$A := \{ u \in \C \, , \, \Gamma_{x_0, y_0} (u) \in \D \times \D \}$$ is an angular domain which does not depend on $(x_0,y_0)$.  Its boundary
  $$\partial A := \{ u \in \C \, ,  \, \vert \Gamma_{x_0,y_0}(u) \vert_\infty = 1 \} , $$
where $\vert (x,y) \vert_\infty = \sup (|x|,|y|)$, is the union of two semi-lines. Let $d_e^A$ and $d_P^A$ denote the euclidian  and the Poincar\'e metric in the angular domain $A$. 
 
 \begin{definition} \label{def:rho}
Let \(\rho \) be a positive smooth function on $S^*$ which satisfies 
 \begin{equation} \label{eq: rho}
 \rho (x,y) = - \log \left(|x|^2+|y|^2 \right) 
\end{equation} 
on $B_p^*$, for every singular point $p$.
\end{definition}

The following lemma will be useful. The second item relies:

- the distance in $S$ between $q \in B_p^*$ and $p$ (given by the function $e^{-\rho})$,

- the euclidian distance $d_e^A$ in $A$ between $u(q)$ and  $\partial B_p \simeq \partial A$.

 \begin{lemma}\label{l: inthedomain} Let $q  = \Gamma_{x_0,y_0}(u) \in B_p^*$ and let $u \in A$.  
 \begin{enumerate}
  \item \(  \log d_e^A (u , \partial A ) \sim d_P^A (u, \partial A)\),
   \item \( \rho (q) \sim d_e^A ( u  , \partial A ) \). 
\end{enumerate}
 \end{lemma}
 
\begin{proof}
The first item is obtained by uniformizing the angular domain $A$ by the upper half plane. The second one comes from the parametrizations $\Gamma_{(x_0,y_0)}$ which involve the exponential function, see Equation \eqref{eq: parametrization}.
\end{proof}

\subsection{The metric $m$ on the normal bundle $N_\FF$}\label{normalbundle}
The normal bundle $N_{\FF}$ is defined by \( TS / \chi(T_\FF) \) over $S^*$, it extends as a line bundle on $S$ in the same way as for the tangent bundle $T_\FF$. Near a singular point, $V$ takes the form $a x \partial _x + b y \partial _y$, see Equation \eqref{eq: V}. Hence the \(1\)-form 
\begin{equation}\label{eq: equation for F} \omega =  a x dy - b y dx  \end{equation}
vanishes exactly on \(T_\FF\) in \(S^*\) and  defines a section of the dual normal bundle \( N_{\FF} ^* \) over \(S^*\)  that does not vanish. 
This section extends to a section of \(N_{\FF}^*\) (still denoted $\omega$) over $S$ and does not vanish there as well. In particular, near a singular point, a  metric \( m  \) on \( N_{\FF} \) has the following expression 
\begin{equation}\label{eq: smooth metric on NF} m := e^\psi |\omega| , \end{equation}
where \(\psi\) is a smooth function (including at the singularity). The curvature form \( \Theta_{m} \) of a metric $m$ on $N_\FF$ is locally defined by \(- \frac{i}{\pi}\partial \overline{\partial} _{\FF} \log m (s) \) where \( s \) is a non vanishing holomorphic section of $N_\FF$.

\begin{remark}\label{vanishcurvNF}
For the remainder of the article, we work with a metric $m$ on $N_\FF$ which satisfies \( m = \vert \omega \vert \) near the singular set. In particular, the curvature of $m$ vanishes in the neighborhood of  \(Sing(\mathcal F)\).
\end{remark}

In the regular part of $\FF$, the metric $m$ can locally be written
\begin{equation}\label{eq: varphi}  m = e^{\varphi_m}  | dt|  ,\end{equation}  
where \( t \) is a submersion defining $\FF$. If $p,q$ are in the same plaque, the norm with respect to $m$ of the derivative of the holonomy map $h_{p,q} : (\Sigma_p , p) \to (\Sigma_q , q)$ from a transversal $\Sigma_p$ at $p$ to a transversal $\Sigma_q$ at $q$ is given by
\begin{equation}\label{eq: holodi}
 \log \vert D h_{p,q} \vert_m = \varphi_m(q) - \varphi_m(p) .
\end{equation}
When passing from a foliation box to another one, the function $\varphi_m$ is modified by adding a leafwise constant function, hence any leafwise derivative of \(\varphi_m\) defines a global form. This yields the following definition.

\begin{definition}
We define on $S^*$ the leafwise \(1\)-form 
\begin{equation}\label{eq: eta} \eta_m := d_{\FF} \varphi_m .\end{equation}
\end{definition}

Near a singular point, there exists a holomorphic $1$-form $\Omega$ such that $d \omega= \Omega \wedge \omega$. Even if $\Omega$ is not unique, its restriction to $T_\FF$ is a well defined section of $K_\FF$, whose expression in the linearizing coordinates is given by 
\begin{equation}\label{eq: alphaomega}
 \Omega_{\vert \FF} =\frac{a+b}{a} \frac{dx}{x} \Big \vert_ \FF = \frac{a+b}{b} \frac{dy}{y}  \Big \vert_ \FF .
 \end{equation}
 
\begin{lemma}\label{l: dvarphi} Let $m$ be a metric on $N_\FF$ satisfying  \( m = \vert \omega \vert \) near the singular set, as in Remark \ref{vanishcurvNF}. Then $\eta_m = \Re \Omega_{\vert \FF}$ near the singular set.
\end{lemma}

\begin{proof}
The form $\omega$ induces by restriction to any transversal to $\FF$ a non vanishing holomorphic $1$-form. In particular it induces a translation structure whose charts are the local primitive of $\omega$. Now observe that $\vert \omega \vert$ is the euclidian metric associated to the translation structure induced by $\omega$, and recall that $m = \vert \omega \vert$ near the singular point. The derivative of the holonomy map $h_\gamma : (\Sigma_p , p) \to (\Sigma_q , q)$ along a leafwise smooth path $\gamma$ between $p$ and $q$ is given, with respect to this structure, by 
$$ Dh_\gamma  = \exp  \int_\gamma \Omega_{\vert \FF} .$$
Hence $\log \vert Dh_\gamma \vert_m =  \int_\gamma   \Re \Omega_{\vert \FF} = \varphi_m(q) - \varphi_m(p)$ by Equation \eqref{eq: holodi}. We deduce that $ \Re \Omega_{\vert \FF} = d_\FF \varphi_m $ near the singular set, as desired. \end{proof}

\subsection{The metrics $g_s$ and $g_P$ on the tangent bundle \( T_\FF \)} \label{ss: gsgp}

\subsubsection{Smooth metrics}\label{smoothmetric}  A \emph{smooth metric} $g_s$ on $T_\FF$ is a smooth hermitian metric on this line bundle. This is a metric on $T_\FF$ over $S^*$ in the usual sense, with the additional property that at the neighborhood of a singularity, the function \( \log g_s (V) \) extends to a non vanishing smooth function, where $V = a x \partial _x + b y \partial _y$ as in Equation \eqref{eq: V}. A leaf of $\FF$ equipped with the  smooth metric is a complete Riemann surface. 

\begin{remark}\label{vanishcurv}
In the remainder of the article, we work with a fixed smooth metric $g_s$ which satisfies \( g_s (V)=1 \). In particular, the curvature of this metric is vanishing on $B_p$ for every  \(p \in Sing(\mathcal F)\).
\end{remark}

\begin{lemma}\label{eq: bound eta smooth}
Let $g_s$ be the smooth metric on \( T_\FF \) defined by Remark \ref{vanishcurv} and let $m$ be the metric on $N_\FF$ defined by Remark \ref{vanishcurvNF}. Let $\eta_m$ be the leafwise $1$-form defined by Equation \eqref{eq: eta}. Then $\vert \eta_{m} \vert _{g_s}$ is bounded from above and below (away of zero) on $S$. 
 \end{lemma}

\begin{proof}
By Equation \eqref{eq: alphaomega}, $\Omega_{\vert \FF}$ has constant coefficients near the singular set, hence is bounded with respect to $g_s$. Lemma  \ref{l: dvarphi} gives $\eta_{m} = \Re \Omega_{\vert \FF}$. 
\end{proof}

\subsubsection{Poincar\'e metric} \label{ss: Poincare metric} We assume  that $\FF$ does not support any foliated cycle (see Section \ref{ss: harmonic} for the definition), hence every leaf $\mathcal L$ is hyperbolic. The Poincar\'e metric  on $\D$ is defined by  
\[  g_\D := 2 \frac{|dz|} {1-|z|^2} . \] 
We define $g_P := \phi_* g_\D$ where  $\phi : \D \to \mathcal L$ is any uniformization, this is a hermitian metric on $\mathcal L$. Let $d_P$ denote the associated leafwise distance. A computation yields $\text{vol}_{g_P} = -2\pi \, \Theta_{g_P}$, where   \(  \Theta_{g_P} - \frac{i}{\pi}\partial \overline{\partial} _{\FF} \log g_P \).  

\subsubsection{Relation between smooth metrics and the Poincar\'e metric}  \label{ss: Poincare metric2}

Let \( g_s\) be the smooth metric on \( T_\FF \) defined by Remark \ref{vanishcurv}. Let \(\xi\) be the function on $S^*$ defined  by
\begin{equation} \label{eq: def phi} g_P= e^{-\xi} g_s ,
 \end{equation}  
it has leafwise continuous derivatives \cite{Verjovsky, Candel_uniformization, Brunella}. 

\begin{lemma}\label{PTM}  Let $r = d (\cdot, Sing(\FF))$ and $\rho \simeq - \log r$ as in Definition \ref{def:rho}. 
\begin{enumerate}
\item $\xi = \log \rho + O(1)$, 
\item in particular, $g_P \simeq \rho^{-1}  g_s$.
\end{enumerate}
\end{lemma}

\begin{proof}
Let $g_S$ be the leafwise restriction of a smooth metric on the compact complex surface $S$. We recall that every singular point of $\FF$ is hyperbolic. Dinh-Nguyen-Sibony proved \cite[Proposition 3.3]{DNS_Entropy II} that 
\[ g_S = r \cdot |\log r| \cdot e^{O(1)} g_P = r \rho \cdot e^{O(1)} g_P \]
near every singular point. Now $g_S$ and $g_s$ differ by a factor that is the length of the tangent vector field $V$, namely
$g_S= g_S(V) g_s$. Since $g_S(V)= r \cdot e^{O(1)}$, we finally get $g_P =  \rho^{-1} e^{O(1)} g_s$, as desired.
\end{proof}

We shall need the following estimates. 

\begin{lemma} \label{l: estimates on drho} Near the singular set, we have
\begin{enumerate}
\item $ \left| d_\FF \rho  \right|_{g_{P}}  = O ( \rho)$.
\item $ \left| d_\FF \xi  \right|_{g_{P}} = O (1)$.  
\end{enumerate}
\end{lemma}

\begin{proof} 
In $B_p$, a computation gives 
\[ d \rho =  - \frac{ \Re  \left( xd\overline{x}+ y d\overline{y} \right) }{|x|^2+|y|^2} .\]
We get \(\left|  d_\FF \rho (V) \right| =O(1) \) where $V = a x \partial _x + b y \partial _y$ is the section of \( T_\FF\) defined by Equation \eqref{eq: V}. Since $g_s(V) = 1$, we have  \( \left| d_\FF \rho \right|_{g_{s}}  =O(1)\). The first item then follows from the second item of Lemma \ref{PTM}. 

For the second item, Equation \eqref{eq: def phi} and the definition of $g_P$ provide  
$$ - {1 \over 2\pi} \text{vol}_{g_P} = \Theta_{g_P} = \frac{i}{\pi}\partial \overline{\partial}_\FF \xi + \Theta_{g_s} . $$
By Remark \ref{vanishcurv}, the curvature $\Theta_{g_s}$ is vanishing on $B_p$, hence $2i  \partial \overline{\partial}_\FF \xi = - \text{vol}_{g_P}$ there. We deduce from  $2i  \, \partial \overline{\partial}_\FF \xi  = \Delta_{g_P} \xi \cdot \text{vol}_{g_P}$ that 
\begin{equation}\label{xiharm}
 \Delta_{g_P} \xi  = -1 \textrm{ on } B_p. 
 \end{equation}
Now let $a \in B_p$ such that the leafwise ball $\mathcal B_a(1)$  of radius \(1\) and center \(a \) for the Poincar\'e metric is included in $B_p$. We use notions introduced in Section \ref{sub:singpoints}, in particular $d_e^A$ stands for the euclidian distance in the angular domain $A \subset \C$. By definition of the Poincar\'e metric, there exists $c \geq 1$ independent from $a$ such that for every $q \in \mathcal B_a(1)$, we get (denoting $u(s)$ the parameter $u \in A$ corresponding to $s \in S^*$):
\begin{equation}\label{disdist}
c^{-1} \,  d_e^A (u(a), \partial A) \leq d_e^A (u(q), \partial A) \leq c \, d_e^A (u(a), \partial A) . 
   \end{equation}
Lemma \ref{l: inthedomain} and the first item of Lemma \ref{PTM}  respectively provide $\rho (q) \sim d_e^A ( u(q) , \partial A)$ and $\xi = \log \rho + O(1)$. Hence, by Equation \eqref{disdist}, the function \( \xi - \xi (a)\) is bounded on $\mathcal B_a(1)$ by a constant which does not depend on $a$.  
 
To finish we use classical Cauchy's estimates for harmonic functions, see \cite[Chapter 2]{ABW}.
Let $\phi : \D \to \mathcal L_a$ be a uniformization such that $\phi(0)=a$. 
Let $r_0  \in ]0,1[$ be such that $\D(r_0) = \phi^{-1} \mathcal B_a(1)$. Equation \eqref{xiharm} implies 
$$\Delta_{g_\D} \xi \circ \phi  = -1 \textrm{ on } \D(r_0) .$$
 Let $L$ be the function $- \log (1 - \vert z \vert^2)$ on $\D$, it is bounded  on $\D(r_0)$,  satisfies $\Delta_{g_\D} L  = 1$ on $\D$, and $\vert dL(0) \vert_{g_\D} = 1$ (up to multiplicative constants). Let us set $h:= \xi \circ \phi - \xi \circ \phi(0) + L$. Then $h$ is harmonic on $\D$ and (by the previous paragraph) is bounded on $\D(r_0)$ by a constant which does not depend on $a$. Cauchy's estimates imply that $\vert dh(0) \vert_{e}$ is bounded  by a constant which does not depend on $a$, where $\vert \cdot \vert_e$ denotes the norm with respect to the euclidian metric on $\D$. To conclude, we observe that 
$$ \vert d_\FF \xi (a) \vert_ {g_P} =\vert d (\xi\circ \phi) (0) \vert_ {g_\D} \leq \vert  d h(0) \vert _{g_\D} + \vert dL(0) \vert_{g_\D} = \vert dh(0) \vert_{e} + 1 , $$
 where the last equality uses the fact $g_\D(0) = g_e(0)$.   \end{proof}

\section{Harmonic currents}

\subsection{Directed harmonic currents} \label{ss: harmonic} A \emph{current} $T$ on \(S\) is a continuous linear form on the space \(\Omega^{1,1}(S)\) of smooth \((1,1)\)-forms. It is \emph{closed} if \( T (d\xi) = 0 \) for every 1-form $\xi$ on $S$, it is \emph{harmonic} if 
 \(T ( i \partial \overline{\partial} f ) = 0 \) for every smooth function on $S$. 
Following \cite{Ghys},  the intersection of a harmonic current $T$ with a line bundle \( E\) over $S$ is defined as follows,  where \(\Theta_h = - dd^c_\FF \log h  \) is the curvature form of a hermitian metric \(h  \) on \(E\) (see Section \ref{clanotion}):
\begin{equation} \label{intercurrents}
 T \cdot E := \int _S \Theta_h \wedge T  .
  \end{equation}
 
Let us  define \emph{directed currents}. Let  \( \mathcal P_\FF \) denote the closed subset of $TS$ given by the union of the tangent spaces of $\FF$ over the regular part $S^*$ and of the tangent spaces of $S$ over the singular part. A $(1,1)$-form $\eta$ is \emph{$\mathcal P_\FF$-positive} if $\eta(u, iu) \geq 0$ for every $u \in \mathcal P_\FF$. Then a current  $T$ is \emph{directed} if $T(\eta) \geq 0$ for every $\mathcal P_\FF$-positive $(1,1)$-form $\eta$, see for instance \cite[Section 2.4]{DDK}.

Skoda \cite{Skoda} proved that if \( T\) is a positive (in particular directed) harmonic current defined at the neighborhood of the origin in \({\bf C}^2\), then 
\begin{equation} \label{eq: monotonicity trace measure} r\mapsto \frac{1}{r^2}  \int_{|x|^2 + |y|^2 \leq r^2}  T \wedge i (dx \wedge d\overline{x} + dy \wedge d\overline{y} ) \end{equation}
is non decreasing. The limit when \(r\) tends to \(0\) therefore exists (it is the Lelong number  of \(T\) at the origin, it does not depend on the coordinate $(x,y)$). In particular, a directed harmonic current does not put any mass on points, and from Equation \eqref{intercurrents}, and  we get for every line bundle $E$ on $S$:
\begin{equation} \label{intercurrents2}
 T \cdot E := \int _S \Theta_h \wedge T = \int _{S^*} \Theta_h \wedge T .
 \end{equation}

Berndtsson-Sibony \cite[Theorem 1.2]{BS} established the existence of directed harmonic currents for singular foliations. Dinh-Nguyen-Sibony  proved in \cite[Theorem 1.1]{DNS_unique ergodicity} that any foliated compact K\"ahler surface with hyperbolic singularities and no foliated cycle supports a unique directed harmonic current, up to a multiplicative constant. Fornaess-Sibony \cite[Theorem 4]{FS}  proved that property on $\mathbb P^2$.  A \emph{foliated cycle} is a closed directed current  \cite{Sullivan}. By Ahlfors' lemma, entire leaves produce foliated cycles \cite{Brunella courbes entieres}, hence every leaf is hyperbolic if there is no foliated cycle. 

\subsection{The measure $\mu = T \wedge \text{vol}_{g_P}$} \label{ss: harmonicm} Let $T$ be a directed harmonic current on $S$. We define on $S^*$ the Radon measure 
\begin{equation}\label{eq: harmonic measure} \mu := T \wedge \text{vol}_{g_P}  \end{equation}
by \( \mu(f) := T ( f \text{vol}_{g_P}) \) for every \( f  \in C^\infty_c (S^*) \).  
The following integrability property, due to Nguyen, is crucial in our work. 

\begin{theorem} \cite{Nguyen} \label{th: Nguyen estimates} 
Let $(S,\FF)$ be a foliated algebraic complex surface with hyperbolic singularities and no foliated cycle. Let $T$ be a directed harmonic current, it is unique up to a multiplicative constant by \cite{DNS_unique ergodicity}. Let $\mu$ be the Radon measure on $S^*$ defined by Equation \eqref{eq: harmonic measure} and let \(\rho \) be a smooth positive function on $S^*$ with a logarithmic pole at every singular point, as in Equation \eqref{eq: rho}. Then $\int_{S^*} \rho \ d\mu$ is finite.
\end{theorem}
As a consequence, the measure \(\mu\) is finite on $S^*$. We extend it to $S$, keeping the same notation $\mu$, by putting no mass on the singular set. We shall see in Section \ref{ss: harmonicm2} that the measure $\mu$ is harmonic. We get the following Lemma from $\mu = T \wedge \mathrm{vol}_{g_P}$, $\Delta_{g_P} \cdot \text{vol}_{g_P}  = 2\pi \, dd^c$ and  Equation \eqref{intercurrents}.

\begin{lemma}\label{forumlaTNF} Let $m$ be a metric on $N_\FF$ as in Remark \ref{vanishcurvNF}, we locally have $m = e^{\varphi_m}  | dt|$ as in Equation \ref{eq: varphi}. Then $\int_{S} \Delta_{g_P} \varphi_m \, d\mu = -2 \pi \, T\cdot N_{\mathcal F}$. 
\end{lemma}

\subsection{Local expression of directed harmonic currents} \label{ss: local harmonic}

A harmonic current is directed if and only if $T \wedge \Omega = 0$ for any holomorphic $1$-form locally defining $\FF$, see for instance \cite[Lemma 2.5]{DDK}. Let $T$ be a directed harmonic current and let \( (z,t) \in {\bf D} \times {\bf D} \) be coordinates of a foliation box. By \cite{Ghys}, there exist a Radon measure $\nu$ on \({\bf D}\) and a non negative function \(\tau \in L^1_{loc} (Leb_{\bf D} \otimes \nu )\), harmonic on \(\nu\)-almost every plaque, such that 
 \begin{equation} \label{eq: local expression harmonic current} \forall \,  (1,1)-\textrm{form} \,  \omega  \ , \ T (\omega )=   \int_{\bf D} \left( \int _{{\bf D} \times \{t\} } \tau \, \omega \right) \nu (dt) . \end{equation}
 
Given a $1$-form $\eta$ on $S$, an integration by parts along the leaves yields $T(d\eta) = T( d_{\mathcal F} \log \tau \wedge \eta)$ and thus the following lemma.

\begin{lemma} \label{dtdt} 
 We have \( d T =   T  \wedge d_{\mathcal F} \log \tau  \) and  \( d ^c T =  T \wedge d_{\mathcal F}^c  \log \tau  \). 
 \end{lemma}
 
Following  \cite[Section 5]{Candel}, we put the next definition.  
\begin{definition}
We define on $S^*$ the leafwise $1$-form
\begin{equation}\label{eq: beta}  \beta := d_{\FF} \log \tau  . \end{equation}
\end{definition}
 
Harnack inequality for positive harmonic functions implies for every $(z,w) \in \bf D \times \bf D$ and for \(\nu\)-almost every $t \in \bf D$:
\begin{equation}\label{eq: harnack}
e^{-d_P(z,w)} \, \tau(w,t) \leq \tau(z,t) \leq e^{d_P(z,w)}  \tau(w,t) ,  
\end{equation}
where $d_P$ is the distance induced by the Poincar\'e metric (up to a multiplicative constant), see Section \ref{ss: Poincare metric}. 

\begin{lemma} \label{lemmaharnack} Let $\tau$ be the density function of $T$ in a foliation box as in Equation \eqref{eq: local expression harmonic current}. Then the forms \(  d_{\mathcal F}  \log \tau\), \(  d^c_{\mathcal F}  \log \tau\) and \(dd^c_{\mathcal F} \log \tau \) are bounded with respect to the Poincar\'e metric $g_P$. 
 \end{lemma}
 
 \begin{proof}
Harnack inequality \eqref{eq: harnack} implies $$ \vert \log \tau(z,t) - \log \tau(w,t)\vert \leq  d_P(z,w) ,$$
which proves the statement for \(  d_{\mathcal F}  \log \tau\) and \(  d^c_{\mathcal F}  \log \tau\).
Now observe that \(dd^c_{\mathcal F} \log \tau = ({\Delta \tau \over \tau} + \vert{ \nabla \tau \over \tau} \vert^2) {i \over 2} dz \wedge d\bar z \), which is equal to $\vert{ \nabla \tau \over \tau} \vert^2 {i \over 2} dz \wedge d\bar z \) because $\tau$ is leafwise harmonic.  Since  \(  d_{\mathcal F}  \log \tau\) is bounded with respect to the Poincar\'e metric, we deduce that the same property holds for \(dd^c_{\mathcal F} \log \tau \).
\end{proof}
 
In Section \ref{ss: restriction of T}, we shall need to integrate $d^cT$ on certain real 3-manifolds. For that purpose, using \( d^c T =   T  \wedge d^c_{\mathcal F} \log \tau  \) provided by Lemma \ref{dtdt} and Equation \eqref{eq: local expression harmonic current2}, we can write
 \begin{equation} \label{eq: local expression harmonic current3}   d^cT = \int_{\bf D} \left( \int _{{\bf D} \times \{t\} } \tau (z, t) \,   ( d^c_{\mathcal F} \log \tau \wedge dt \wedge d\overline{t}) \, Leb_{\bf D} (dz)  \right) \nu (dt) . \end{equation}
If $M$ is a real \(3\)-dimensional manifold \( M\subset S^* \) transverse to $\FF$, we define \(\int _M d^cT\) by integrating on $M$ the measurable $3$-form given by Equation \eqref{eq: local expression harmonic current3}. This definition is extended to $3$-forms of the type $T \wedge \eta$ instead of $d^c T$. 

\subsection{Computation of the intersection $T \cdot T_{\mathcal F}$} \label{ss: restriction of T}

Our aim is to prove the following result, which will be used to show Proposition \ref{d: lyapunov exponent}.

\begin{proposition} \label{p: cohomological interpretation of the mass of the harmonic measure}
Let $(S,\FF)$ be a foliated algebraic complex surface with hyperbolic singularities and no foliated cycle. Let $T$ be a directed harmonic current. Then 
 $$ T \cdot T_{\mathcal F} = - {1\over 2\pi} \int _S T \wedge \text{vol}_{g_P}  .$$
 In particular, $T \cdot T_{\mathcal F}  = -{1 \over 2\pi} \mu(S)$ by Section \ref{ss: harmonicm}.
\end{proposition}

\begin{proof} By Equation \eqref{intercurrents2}, we have $T \cdot T_{\mathcal F}   =  \int_{S^*} T \wedge \Theta_{g_s}$, where $g_s$ is the smooth metric on $T_\FF$ fixed in Section \ref{smoothmetric}.
With the notations of Section \ref{ss: Poincare metric2}, we have $g_P= e^{-\xi} g_s$ and  then (as in the proof of Lemma \ref{l: estimates on drho})
$$ -{1 \over 2\pi} \text{vol} _{g_P} = \Theta_{g_P} = \frac{i}{\pi}\partial \overline{\partial}_\FF \xi + \Theta_{g_s} . $$ 
The curvature  \(\Theta_{g_s} \) vanishes near the singular set and $T \wedge \text{vol}_{g_P}$ is finite on $S^*$ by Nguyen's theorem (see Section \ref{ss: harmonicm}). We get that
 \( i\partial \overline{\partial}_\FF \xi \wedge T\) has finite mass on $S^*$. It thus suffices to establish
\begin{equation} \label{eq: vanishing of integral} \int _{S^*} i\partial \overline{\partial}_\FF  \xi \wedge T= 0.\end{equation} 
Let us write Green's formula on the complementary of the union of balls of radius \(r_0\) around the singular points of \(\mathcal F\) in the \((x,y)\)-coordinates, namely on the domain \(B_{r_0}^c:= \{  \rho \leq - \log r_0  \} \). If \(r_0\) is sufficiently small, the boundary of \(B_{r_0}^c\) is transverse to the foliation. Since $T$ is a directed harmonic current, 
\begin{equation}\label{eq:STO}
  \int_{B_{r_0}^c} dd_\FF^c \xi \wedge T = \int _{\partial B_{r_0}^c} \xi d_\FF^c T - d_\FF^c \xi \wedge T  . 
  \end{equation} 
This formula is obtained by introducing a covering of $B_{r_0}^c$ with a partition of unity, and applying Green's formula on the plaques of the foliation boxes.  
 
Let us estimate the two terms of the right hand side of Equation \eqref{eq:STO}. Let \( dl _P\) denote the length element of the trajectories of \( \mathcal F \cap  \partial B_{r_0}^c \) with respect to the Poincar\'e metric.
On one hand, Item 2 of Lemma \ref{l: estimates on drho} implies 
$$ \left| \int_{\partial B_{r_0}^c}  d_\FF^c \xi \wedge T \right| \leq c \int_{\partial B_{r_0}^c}  dl_P \wedge T . $$
On the other hand, we first have (by $\xi = \log \rho + O(1)$, see Lemma  \ref{PTM})
$$ \left| \int_{\partial B_{r_0}^c} \xi d_\FF^c T \right|  \leq  \log \left( -  \log r_0 \right) \left| \int_{\partial B_{r_0}^c} d_\FF^c T \right| .$$
Lemmas \ref{dtdt} and \ref{lemmaharnack} then yield $\left| \int _M d_\FF^c T \right| \leq  c \, \int_M  T  \wedge dl_P$, hence
$$  \left| \int_{\partial B_{r_0}^c} \xi d_\FF^c T \right|  \leq  c  \log \left( -  \log r_0 \right)  \int_{\partial B_{r_0}^c}  dl_P \wedge T . $$
Taking into account these two upper estimates, we get 
\begin{equation} \label{eq: estimates on the integral of ddc varphi} \left| \int_{B_{r_0}^c} dd_\FF^c \xi \wedge T \right| \leq c \log \left( -  \log r_0 \right) \ \int_{\partial B_{r_0}^c}  dl_P \wedge T .
\end{equation}
Since \( i\partial \overline{\partial}_\FF \xi \wedge T\) has finite mass on $S^*$, the left hand side of Equation \eqref{eq: estimates on the integral of ddc varphi} tends to the modulus of $\int _{S^*} dd_\FF^c \xi \wedge T$ when $r_0 \to 0$. Let us verify that the $\liminf_{r_0 \to 0}$ of the right hand side of Equation \eqref{eq: estimates on the integral of ddc varphi} tends to zero, that will establish Equation \eqref{eq: vanishing of integral}. 

Let us denote \( f:= \log \rho\). It suffices to prove
\begin{equation}\label{gogoal}
\liminf_{u\rightarrow +\infty} u  \int_{f=u } dl_P \wedge T = 0 .
\end{equation}
Indeed, by setting \( u = \log (- \log  r_0) \) and noting that $\{ f = u\} = \{ \rho = - \log r_0 \} = \partial B_{r_0}^c$, we obtain $$\liminf_{r_0 \rightarrow 0}   \log \left( -  \log r_0 \right)  \int_{\partial B_{r_0}^c} dl_P \wedge T = 0 . $$ 
Now we prove  Equation \eqref{gogoal}. For  \(u_0\) large enough so that \(\{f=u\}\) is transverse to \(\mathcal F\) for any \(u\geq u_0\), the coarea formula \cite[Theorem 3.2.22]{Federer} yields  
\[ \int _{u_0} ^\infty \left( \int_{f=u } dl_P \wedge T  \right) du =\int _{f \geq u_0 } \vert d_\FF f \vert_{g_P} \text{vol}_{g_P} \wedge T . \]
Now  $\vert d_\FF f  \vert _{g_P} = O(1)$ by the first Item of Lemma \ref{l: estimates on drho}. We deduce 
\[ \int _{u_0} ^\infty \left( \int_{f=u } dl_P \wedge T  \right) du \leq cst  \int _{f\geq u_0} \text{vol}_{g_P} \wedge T  <\infty.  \]
This implies Equation \eqref{gogoal}, as desired. 
\end{proof}

\part{A Limit Theorem for leafwise closed $1$-forms} \label{ss: action holo}

\section{Leafwise continuous paths} 

\subsection{Holonomy associated with first integrals} \label{ss: holonomy} 

Given a leafwise continuous path  \(\gamma: [a,b] \rightarrow L\) and two germs of first integrals \(t : (S, \gamma (a) ) \rightarrow ({\bf C},  t(\gamma(a)) ) \) and \(t' : (S,\gamma(b)) \rightarrow ({\bf C},  t'(\gamma(b)) )\), one can associate a germ of biholomorphism 
$$ h_{\gamma, t, t'} : ({\bf C}, t(\gamma(a))) \rightarrow ({\bf C},  t'(\gamma(b))) $$  called the  \emph{holonomy} and defined as follows. There exists a subdivision \( c_0= a <c_1 <\ldots < c_r = b\) and open subsets $  ( U_i )_{i=0, \ldots, r-1} $ on which a first integral \( t_i : U_i \rightarrow {\bf C}\) exists, and such that \( \gamma([c_i, c_{i+1} ]) \subset U_i\). At the neighborhood of \( \gamma(c_{i})\), the first integral \( t_i\) is a function of \(t_{i-1}\), namely, there exists a germ of biholomorphism \( h_i : ({\bf C} , t_{i-1} (\gamma(c_i))) \rightarrow ({\bf C}, t_i (\gamma(c_i))) \) such that \( t_i = h_i \circ t_{i-1}\). The holonomy map is defined by 
\[ h_{\gamma, t, t'} := h_{r-1} \circ \ldots \circ h_1 .\]
It does not depend on the choices, and only depends on the homotopy class with fixed extremities of \(\gamma\) in its leaf, see \cite[Section 2.3]{Candel-Conlon}.

\subsection{Holonomy associated with transversals} \label{ss: holonomy2} 

Let \( \Sigma \) and \( \Sigma' \) be local transversals at \(\gamma(a)\) and \(\gamma(b)\) respectively. The first integrals induces germs of biholomorphisms \( t : ( \Sigma , \gamma(a) ) \rightarrow ({\bf C}, t(\gamma(a)))\) and  \( t' : ( \Sigma' , \gamma(b) ) \rightarrow ({\bf C}, t'(\gamma(b)))\). The composition \( h_{\gamma, \Sigma,\Sigma'} := ( t')^{-1} \circ h_{\gamma, t, t'} \circ t \) defines a germ of holomorphic map \( (\Sigma, \gamma(a)) \rightarrow (\Sigma', \gamma(b))\) that does not depend on the first integrals \(t,t'\). 

\subsection{The dynamical system \( ( \OO, \sigma, W_{g_{P}}^\mu) \)}\label{ss: harmonicm2}

Let $p(t,x,y)$ be the leafwise heat kernel for the Poincar\'e metric $g_P$. It is the fundamental solution of the heat equation $\partial_t = \Delta_{g_P}$, with $\lim_{t \to 0} p(t,x,y) = \delta_x(y)$ in the sense of distributions, where $\delta_x$ is the Dirac mass at $x$. Let \(\OO \) be the set of leafwise continuous paths \(\oo : [0,+\infty) \rightarrow S^*\) and \( \{\sigma_{\Timet} \}_{\Timet \geq 0} \) be the shift semi-group acting on \(\OO\) by 
$$ \sigma_\Timet (\oo) (\Times) = \oo(\Timet+\Times) , \text{ for all } \Times,\Timet \geq 0. $$
For every \( x \in S^*\), let \(\OO^x \) be the subset of leafwise continuous paths starting at \(\oo(0)= x \). Let \( W^x_{g_{P}} \) be the Wiener measure on \(\OO ^x\) so that $$ \forall \Timet > 0 \ , \ \int_{\Gamma^x} f(\gamma(\Timet)) dW^x_{g_P} (\gamma) = \int_{L_x} f(y) p(x,y,\Timet) \text{vol}_{g_P}(dy) .  $$
The measure \(\mu = \text{vol}_{g_P} \wedge T \) defined in Equation \eqref{eq: harmonic measure} is  \emph{harmonic}, meaning that \(\int \Delta_{g_P} f d\mu =0\) for every \( f \in C^\infty_c (S^*) \), this is due to the formula \(\Delta_{g_P} \cdot \text{vol}_{g_P} = 2i \,  \partial \overline{\partial}_\FF\). 
This implies that the following probability measure on $\Gamma$ 
$$ W_{g_{P}}^{\mu} := \int_{S^*} W_{g_{P}}^x d\mu (x)  $$ 
is \(\sigma\)-invariant. Observe that if \(T\) is extremal in the compact convex cone of directed harmonic currents, then  \(\mu\)  is also extremal in the compact convex set of harmonic measures. Moreover, Garnett's random ergodic theorem \cite{Garnett} states that if \(\mu \) is extremal then \( ( \OO, \sigma, W_{g_{P}}^\mu) \) is ergodic. 

The foliated space \(S^* \) is actually not compact, but Garnett's arguments extend to this case, let us outline the proof. Assume that \( B \subset \OO\) is a $\sigma$-invariant measurable subset, this is a tail event. Let $\Psi(x) := W_{g_{P}}^x(B)$ so that \( W^\mu_{g_P}(B) = \int_{S^*} \Psi(x) d\mu(x)\). Since \(B\) is a tail event, the restriction of \( \Psi \) to each leaf \( \LL \) is a harmonic function. Now, given any \(c \in \mathbb Q\), the function \( \Psi_c : = \sup \{ \Psi, c \} \) is a measurable subharmonic function, namely  \( \Psi_c \leq D^\Timet (\Psi_c) \) for any \(\Timet \geq 0\), where 
$$ D^\Timet (\Psi_c) = \int_\LL p(x,y,\Timet) \Psi_c(y) \text{vol}_{P} (dy) . $$
But $\mu$ is harmonic, hence \( \int_{S^*} D^\Timet(\Psi_c) \, d\mu = \int_{S^*} \Psi_c \, d\mu\), and so \( D^\Timet (\Psi_c)=\Psi_c\) for any \( (\Timet,c) \in \mathbb Q^2\) in restriction to $\mu$-a.e. leaf $\LL$. This implies that \( \Psi\) is constant on $\mu$-a.e. leaf \( \LL\), and since \(\mu\) is extremal, \( \Psi \) is constant \(\mu\)-a.e. Now consider the ordered family \(\{\mathcal M_\Timet\} _{\Timet \geq 0}\) of \(\sigma\)-algebras on \( \OO\) defined by the map that sends  \(\oo \in  \OO\) to its restriction to \([0,\Timet]\).  This is a general fact that for any measurable bounded function \(f\) on \(\OO\), the conditionals \( \mathbb E (f \, | \, \mathcal M_\Timet)\) converges $W^\mu_{g_P}$-a.e. to \(0\) or \(1\). If \( f\) is the characteristic function of \(B\), we have  \( \mathbb E (f \, | \, \mathcal M_\Timet)= \Psi (\gamma(\Timet)) \) by the tail property, so that \(\Psi(\gamma(\Timet))\) converges to \(0\) or \(1\)  for $W^\mu_{g_P}$-a.e. \(\gamma \in \Gamma\). We have thus proved that \(\Psi\) is $\mu$-a.e. equal to $0$ or $\mu$-a.e. equal to $1$. Finally \(W^\mu_{g_P}(B) = \int_{S^*} \Psi(x) d\mu(x) \in \{ 0,1 \} \), as desired.

\subsection{Integrability of the $g_s$-length of leafwise continuous paths}\label{techtech}  Let \(g\) be a metric on \( T_\FF\) over \( S^* \) (namely a smooth metric $g_s$ or the Poincar\'e metric $g_P$). For each \(\Timet\geq 0\), we define    
\begin{equation} \label{eq: cocycle smooth distance} 
 D_\Timet^g : \Gamma \rightarrow {\bf R} \ , \ D_\Timet^g (\gamma ) := \inf l _{g} (\overline{\gamma}) \end{equation}
where the infimum runs over the piecewise smooth paths \(\overline{\gamma}:[0,\Timet] \rightarrow \LL_{\gamma(0)}\) homotopic to \(\gamma_{[0,\Timet]}\) with fixed extremities, $l_g$ denoting the $g$-length. Our aim is to prove Proposition \ref{p: integrability} below. Let us begin with the following result.
 
\begin{lemma}\label{lemma: rough bound for Ht}
There exist \( c, c'>0\) such that for every \(\oo \in \OO\) and \(\Timet\geq 0\), 
$$  D_\Timet^{g_s} (\oo)   \leq c' \rho (\oo(0)) \exp ( c D_\Timet^{g_P} (\oo)) .$$
\end{lemma}

\begin{proof} Let \(\oo\in \OO ^p \) and \(\overline{\oo}: [0, D_\Timet^{g_P}(\oo)]\rightarrow \LL\) be a $g_P$-geodesic parametrized by arc length, which is homotopic with fixed endpoints to the curve \(u\in [0, D_\Timet^{g_P}(\oo) ] \mapsto \oo(\Timet u/D_\Timet^{g_P}(\oo)) \in L\). In particular, 
\[ \overline{\oo} (0)= \oo (0), \ \overline{\oo} ( D_\Timet^{g_P}(\oo) ) = \oo(\Timet)\text{ and } D_\Timet ^{g_P} (\oo) = D^{g_P} _{D_\Timet^{g_P}(\oo)} (\overline{\oo}) .\]
Since $\overline{\oo}$ is parametrized by $g_P$-arc length, we have $g_P((\overline{\oo})'(u)) = 1$ for every $\Timeu \in  [0, D_\Timet^{g_P}(\oo) ]$. Using $g_s \simeq \rho \, g_P$ given the second item of Lemma \ref{PTM}, there exists \( a >0\) such that  
$$ g_s ((\overline{\oo})'(\Timeu)) \leq a \rho (\oo (\Timeu)) g_P((\overline{\oo})'(\Timeu)) = a \rho (\oo (\Timeu)) ,$$
which implies
 \begin{equation}\label{eq: estimates on holonomy} \frac{d D_\Timeu^{g_s}(\overline{\oo}) }{d\Timeu} \leq  a \rho (\overline{\oo} (\Timeu) ) . \end{equation}
We deduce, using the first item of  Lemma \ref{l: estimates on drho} for the second inequality:
\[  \forall  \Timeu \in [0, D_\Timet^{g_P}(\oo) ] \ , \  (\rho \circ \overline{\oo})'(\Timeu)  \leq \vert (\rho \circ \overline{\oo})' (\Timeu)) \vert =  \frac{\vert d (\rho \circ \overline{\oo}) (\Timeu)) \vert}{g_P((\overline{\oo})'(\Timeu))} \leq  b \rho(  \overline{\oo}(\Timeu) )  . \]
Gronwall's inequality then implies  
\begin{equation}\label{eq: estimates on rho} 
 \forall  \Timeu \in [0, D_\Timet^{g_P}(\oo) ] \ , \  \rho ( \overline{\gamma}(\Timeu) ) \leq \rho (\overline {\oo}(0)) \, \exp (b \Timeu ) =  \rho ( {\oo}(0)) \, \exp (b \Timeu ) .
  \end{equation}
Estimates \eqref{eq: estimates on holonomy} and \eqref{eq: estimates on rho} yield $d D_\Timeu^{g_s}(\overline{\oo}) / d\Timeu \leq a \rho (\oo(0)) \, \exp (b \Timeu)$. We finish by integrating over  \( \Timeu \in [ 0 ,  D_\Timet^{g_P}(\oo) ] \) and setting $c:=b$, $c':=a/b$. \end{proof}

\begin{proposition}\label{p: integrability}
Let $S$ be an algebraic surface endowed with a foliation $\FF$ with hyperbolic singularities and no foliated cycle. Let \( T\) be the unique directed harmonic current satisfying \( \int T\wedge \text{vol}_{g_P} = 1\). Let $g_s$ be a smooth metric on \(T_\FF\). For every \(\Timet\geq 0\), the function  \( \sup _{\Timeu \in [0,\Timet]} D^{g_{s}}_\Timeu\) is \(W^\mu _{g_{P}}\)-integrable. 
\end{proposition}

\begin{proof} 
By Lemma \ref{lemma: rough bound for Ht}, we have
$$ \forall \gamma \in \Gamma^x \ , \  \sup _{\Timeu \in [0,\Timet]} D^{g_{s}}_\Timeu (\gamma) \leq c' \rho(x) \sup _{\Timeu\in [0,\Timet]} \exp ( c D_\Timeu^{g_P} (\oo)) .$$
We shall prove below that there exists $c''$ such that  
\begin{equation}\label{eq: heher}
\forall x \in S^* \ , \ \int_{\Gamma^x} \sup _{\Timeu\in [0,\Timet]} \exp (c D^{g_P}_\Timeu (\gamma)) d W^x_{g_P}(\gamma)  \leq c'' .
 \end{equation}
This yields $\varphi(x) := \int_{\Gamma^x} \sup _{\Timeu \in [0,\Timet]} D^{g_{s}}_\Timeu (\gamma) \leq c' c'' \rho(x)$. Hence $\varphi$ is $\mu$-integrable by Theorem \ref{th: Nguyen estimates}, and finally  \( \sup _{\Timeu \in [0,\Timet]} D^{g_{s}}_\Timeu\) is \(W^\mu _{g_{P}}\)-integrable as desired. 

Now let us prove Equation \eqref{eq: heher}. 
The hyperbolic disc is complete with constant curvature. Hence we can use Equation (8.65) in \cite{Stroock} to deduce that for every \( \Timet\geq 0\), there exists a constant \(\delta(\Timet)\) such that  
$$  \forall x \in S^* \ , \ W^x_{g_P} ( \sup _{\Timeu\in [0,\Timet] } \exp (c D^{g_P} _\Timeu(\gamma) )  \geq r ) \leq \sqrt{2} \exp \left(-\frac{\log ^2 r}{4\Timet c^2} + \delta(\Timet)\right) . $$
To conclude note that the right hand side is integrable as a function of \(r\).\end{proof}

\section{Lyapunov exponent and Dynamical entropy}

\subsection{A general Limit Theorem}\label{ss: integral of leafwise closed one forms}  In this section we generalize limit theorems of \cite{Candel, Deroin, NguyenAMS} to our context with singularities. We shall use the integrability property provided by Proposition \ref{p: integrability}. 

Let $T$ be a harmonic current and let $\mu$ be the probability measure $\text{vol}_{g_P} \wedge T$  
as in Equation \eqref{eq: harmonic measure}. Let \(E\subset S^* \) be a saturated measurable subset of full \(\mu\)-measure. For every $\LL \subset E$, let $\alpha_\LL$ be a smooth closed \(1\)-form on $\LL$. We assume that $\alpha_\LL$ depends measurably on $\LL$, that is on the $\nu$-generic transverse parameter (in each foliation box) in the smooth topology.
Let  \(\varphi := \int \alpha_\LL \) denote a  local integral and let \( \theta : E\rightarrow {\bf R}\) be the leafwise continuous function  
$$ \theta   := \Delta_{g_P} \varphi  =d^* _{g_P} \alpha_\LL . $$
Let us introduce the following properties : 
\begin{enumerate}
 \item[(i)] for every leaf \(\LL\subset E\), the norm of \(\alpha_\LL \) with respect to $g_s$  is bounded by a constant $M_1$ independant of the leaf  \(\LL\), 
\item[(ii)] the function \(\theta\) is bounded on $E$ by a constant $M_2$. 
\end{enumerate}

Let \(\Gamma^E \subset \Gamma\) be the subset of leafwise paths \(\oo \in \Gamma\) starting at a point of \(E\), and define  
\begin{equation}\label{def: Ht}
 H_\Timet ^\alpha : \Gamma^E \rightarrow {\bf R} \ , \  H_\Timet ^\alpha (\oo) := \int _{\overline{\oo}_{[0,\Timet]}} \alpha , 
 \end{equation}
where \( \overline{\oo}: [0,\Timet]\rightarrow L_{\oo(0)}\) is any smooth path  homotopic to \(\gamma\) with fixed extremities. 
 We have \( H_{\Timet+\Times} ^\alpha = H_\Timet^\alpha+ H_\Times^\alpha \circ \sigma_\Timet\) for every \(\Times,\Timet\geq 0\). 

\begin{theorem}\label{p: expression of a}
Let $(S, \FF)$ be a foliated algebraic surface with hyperbolic singularities and no foliated cycle. Let \( T\) be the unique directed harmonic current satisfying \( \int T\wedge \text{vol}_{g_P} = 1\). For every $\LL \subset E$, let $\alpha_\LL$ be a smooth closed \(1\)-form on $\LL$. Let 
$$ A := \int_{S^*} \int_{\Gamma^x}  H_1^\alpha(\gamma) dW^x_{g_P} (\gamma) \, d\mu(x) .$$ 
\begin{enumerate}
\item if Property $(i)$ is satisfied, then $A \in \R$ and  
$$ \textrm{for } W_{g_P}^\mu-a.e. \  \gamma \in \OO  \ , \  \lim _{t\rightarrow +\infty} {1 \over t}  H_t^\alpha(\gamma) = A . $$
\item for every $\Timet >0$, $\Timet A = \int_{S^*} \int_{\Gamma^x}  H_\Timet^\alpha(\gamma) dW^x_{g_P} (\gamma) \, d\mu(x)$. 
\item if moreover Property $(ii)$ is satisfied,  then $A = \int _{S^*} \theta \, d\mu$.
\end{enumerate}
\end{theorem}

\begin{proof}
To simplify we denote  $H^\alpha_\Timet$ by $H_\Timet$. By $(i)$,  \( H_\Timet \leq M_1 D^{g_s}_\Timet \) on $\Gamma$ for every \(\Timet > 0\), implying by Proposition \ref{p: integrability} that  \( H_\Timet  \) is \( W^\mu_{g_P}\)-integrable.  Taking into account that  \( \{H_\Timet\}_{\Timet\geq 0} \) is a cocycle and that \( (\Gamma, \sigma, W^\mu_{g_P})\) is ergodic, Kingman's ergodic theorem implies that there exists \( A\in \R \cup \{ - \infty \} \) such that $\lim_{\Timet \to + \infty} {1 \over \Timet}  H_\Timet (\oo)=A$ for \(W^\mu_{g_P}\)-a.e. \(\oo \in \OO\) and $\lim_{\Timet \to + \infty} \int_\Gamma {1 \over \Timet} H_ \Timet \, dW^\mu_{g_P} = A$.

Now we moreover assume $(ii)$ and prove that $A = \int _{S^*} \theta \, d\mu$. Since  \( \{H_\Timet\}_{\Timet\geq 0} \) is an additive cocycle and since $\Timet \mapsto \int_\Gamma H_ \Timet \, dW^\mu_{g_P}$ is continous, we have $\int_\Gamma {1 \over \Timet} H_ \Timet \, dW^\mu_{g_P} = A'$ for every $\Timet > 0$, where $A' := \int_\Gamma H_ 1 \, dW^\mu_{g_P} \in \R$. Taking limits when $\Timet$ tends to infinity, we get $A=A'$.  We thus have 
\begin{equation}\label{aaa}
  \forall \Timet  > 0 \ , \ \int_\Gamma H_ \Timet \, dW^\mu_{g_P} = \Timet A .
  \end{equation} 
By \( W_{g_{P}}^{\mu} = \int_{S^*} W_{g_{P}}^p d\mu (p) \), see Section \ref{ss: harmonicm2}, the left hand side of \eqref{aaa} satisfies  
\begin{equation}\label{bbb}
\int_\Gamma H_ \Timet \, dW^\mu_{g_P} = \int_{S^*} \left( \int_{\Gamma^x} H_\Timet \,  dW _{g_P} ^x \, \right)  d\mu (x) =:  \int_{S^*}  \mathbb E^x (H_\Timet)  d\mu (x), 
 \end{equation}
 where $\mathbb E^x$ denotes the integral on $\Gamma^x$ with respect to $W _{g_P} ^x$ (since \( H_\Timet  \) is \( W^\mu_{g_P}\)-integrable, \( H_\Timet  \) is also \( W^x_{g_P}\)-integrable  for $\mu$-a.e. $x \in S^*$). This proves Items (1) and (2). Now, since \( H_\Timet (\oo) \) only depends on the homotopy class of \(\oo_{[0,\Timet]} \) with fixed extremities, the local integral $\int \alpha_\LL$ induces a smooth function 
$$ \tilde \varphi : \widetilde{\LL_x} \simeq \D \rightarrow \mathbb R $$
on the universal cover  of $\LL_x$ as follows. Let $\tilde x$ be a lift of $x$, and for every $\oo \in \Gamma^x$, let $\tilde{\oo}$ be the lift of $\oo$ starting at $\tilde x$. We then define \(  \tilde \varphi( \tilde{\oo}_\Timet ) := H_\Timet (\oo) \) where $\tilde \gamma_\Timet := \tilde \gamma (\Timet)$. The function \(\oo\mapsto \tilde \varphi (\tilde{\oo}_\Timet) \) remains \(W _{g_P} ^x \)-integrable for \(\mu\)-a.e. \(x \in S^*\) and for every \( \Timet > 0 \). Let us fix $\Timet$ and apply Dynkin's formula  with a stopping time that we now define. Let $B_{\tilde x}(R)$ denote the leafwise ball centered at $\tilde x$ and of radius $R$ for the metric $g_\D$, and let 
$$u_R(\tilde \gamma):= \inf\{ \Times > 0 \, , \, \tilde \gamma_\Times \notin B_{\tilde x}(R) \} \ \ , \ \  v_R(\tilde \gamma) := \inf \{ \Timet \, , \, u_R(\tilde \gamma) \} . $$
The function $v_R$ is a stopping time satisfying $\mathbb E^x (v_R) \leq \Timet$. Dynkin's formula \cite[Lemma 17.21]{Kallenberg} then asserts (recalling that $\theta =  \Delta_{g_P} \varphi$): 
\begin{equation}\label{uuu}
 \mathbb E^x (\tilde \varphi (\tilde{\oo}_{v_R(\tilde \oo)} ) ) = \frac{1}{2} \mathbb E^x \left(\int _0^{v_R(\tilde \oo)} \theta (\oo(\Times)) d\Times  \right)  .  
\end{equation}
Our goal is now to prove 
\begin{equation}\label{vvv}
 \forall \Timet > 0 \ , \ \mathbb E^x (H_\Timet) = \mathbb E^x (\tilde \varphi (\tilde{\oo}_\Timet ) ) = \frac{1}{2} \mathbb E^x \left(\int _0^\Timet \theta (\oo(\Times)) d\Times  \right)  . 
 \end{equation}
We shall use Lebesgue convergence theorem when $R$ tends to infinity. First observe that $\lim_{R} v_R(\gamma)=\Timet$ for \(W _{g_P} ^p \)-almost every $\gamma$. Using that the $g_s$-norm of $\alpha_L$ is bounded by $M_1$ (Property $(i)$), the fact that $v_R \leq \Timet$ and the function $D_\Timeu^{g_s} (\gamma )$ defined in Equation \eqref{eq: cocycle smooth distance}, we get  for \(W _{g_P} ^p \)-almost every $\gamma$:
$$\vert \tilde \varphi (\tilde \gamma_{v_R(\tilde \gamma)}) \vert = H_{v_R(\tilde \gamma)}(\tilde \gamma) \leq M_1 \sup \{ D_\Timeu^{g_s} (\gamma ) , \Timeu \in [0,\Timet] \} . $$
By Proposition \ref{p: integrability}, for $\mu$-almost every $x$, the right hand side is a \(W^x _{g_{P}}\)-integrable function. Hence, $\tilde \varphi$ being continuous, the left hand side of Equation \eqref{uuu} tends to $\mathbb E^x (\tilde \varphi (\tilde{\oo}_\Timet ) )$ when $R$ tends to infinity. The right hand side of Equation \eqref{uuu} tends to $\frac{1}{2} \mathbb E^x (\int _0^\Timet \theta (\oo(s)) ds $: indeed, using that $\theta =  \Delta_{g_P} \varphi$ is bounded by $M_2$  (Property $(ii)$) and that $v_R \leq \Timet$, we get the domination  
$$ \Big \vert \int _0^{v_R(\tilde \oo)} \theta (\oo(\Times)) d\Times \Big \vert \leq  M_2 \, \Timet  \in L^1 (W^x _{g_{P}}) . $$
This proves Equation \eqref{vvv}. Integrating this Equation with respect to $\mu$ and then using Equations \eqref{aaa} and \eqref{bbb}, we get
\begin{equation} \label{eq: lyapunov exponent with Dynkin formula} \forall \Timet > 0 \ , \  A = \frac{1}{2} \int_{S^*} \mathbb E^x \left( \frac{1}{\Timet} \int _0^\Timet \theta (\oo(\Times)) d\Times   \right) d\mu(x) . \end{equation} 
Since $\theta$ is bounded by $M_2$ and leafwise continuous,  $x \mapsto \mathbb E^x ( \frac{1}{\Timet} \int _0^\Timet \theta (\oo(\Times)) d\Times )$ is also  bounded by $M_2$ and converges to $\theta(x)$ when $\Timet \to 0$ for $\mu$-almost every $x \in S^*$. This implies $A = \frac{1}{2} \int_{S^*} \theta(p) d\mu(p)$ as desired.\end{proof}

\subsection{Lyapunov exponent $\lambda$} \label{ss: LE}

Let $m$ be a metric on $N_\FF$, locally defined by $m = e^{\varphi_m}  | dt|$ as in Equation \eqref{eq: varphi}. We recall that  $\eta_m$ is the leafwise \(1\)-form on $S^*$ locally defined by $d_{\FF} \varphi_m$. Given a leafwise smooth path \( \gamma \in \Gamma \) and transversals \( \Sigma_0, \Sigma_\Timet \) at $\gamma(0)$, $\gamma(\Timet)$, we consider $h_{ \gamma_{[0,\Timet]}, \Sigma_0, \Sigma_\Timet}$ as in Section \ref{ss: holonomy2}. We have 
\begin{equation} \label{eq: norm of derivative of holonomy}  |  Dh_{ \gamma_{[0,\Timet]}, \Sigma_0, \Sigma_\Timet} ( \gamma(0)) |_m = \exp \int_{ \gamma_{[0,\Timet]}} \eta_m .\end{equation}
The integral of $\eta_m$ can also be defined along continuous leafwise paths, for instance Brownian trajectories (by homotoping such paths with fixed extremities to smooth leafwise ones).
The Proposition below states that  ${1 \over \Timet} \int _{\gamma_{[0,\Timet]}}  \eta_m$ converges when $\Timet \to \infty$ for \(W_{g_P}^\mu\)-a.e. \(\gamma \in \OO\) to a universal number.

\begin{proposition}\label{d: lyapunov exponent}
Let $(S,\FF)$ be a foliated algebraic surface with hyperbolic singularities and no foliated cycle. Let \( T\) be the unique directed harmonic current satisfying \( \int T\wedge \text{vol}_{g_P} = 1\). Let \(m \) be the metric on $N_\FF$ defined by Remark \ref{vanishcurvNF}. There exists $\lambda \in \R$ such that for \(W_{g_P}^\mu\)-a.e. \(\gamma \in \OO\): 
\[ \lim _{\Timet\rightarrow +\infty} {1 \over \Timet} \log \vert Dh_{\oo_{[0,\Timet]}, \Sigma_0, \Sigma_\Timet}(\oo(0)) \vert_{m}  = \lambda  .\]
This number is called the  Lyapunov exponent and satisfies  the formula  
$$ \lambda  =  \frac{T\cdot N_{\FF} }{T\cdot T_{\FF}} . $$
\end{proposition}

\begin{proof}   
Equation \eqref{eq: norm of derivative of holonomy} yields the first equality of the following formula, the second one is given by Equation \eqref{def: Ht} with $\alpha = \eta_m$:
\[ \forall \gamma \in \Gamma \ , \ \log \vert Dh_{\oo_{[0,\Timet]}, \Sigma_0, \Sigma_\Timet}(\oo(0)) \vert_{m} = \int_{{\oo_{[0,\Timet]}}} \eta_m = H_\Timet ^{\eta_m} (\oo) . \]
By Lemma \ref{eq: bound eta smooth},  \( \eta _m\) is bounded with respect to $g_s$, hence satisfies Property $(i)$  of Section \ref{ss: integral of leafwise closed one forms}.  Property $(ii)$ comes from the fact that \(\theta := d^*_{g_P} \eta_m  = \Delta_{g_P} \varphi_m \) is smooth on \(S^*\) and vanishes near the singular set. Theorem \ref{p: expression of a} implies 
\[  \textrm{for } W_{g_P}^\mu-a.e. \  \gamma \in \OO  \ , \   \lim _{\Timet\rightarrow +\infty} {1 \over \Timet} \log \vert Dh_{\oo_{[0,\Timet]}, \Sigma_0, \Sigma_\Timet}(\oo(0)) \vert_{m}  = \int_{S^*}  \Delta_{g_P} \varphi_m \, d\mu   .\]
Setting $\lambda  :=  \int_{S^*}  \Delta_{g_P} \varphi_m \, d\mu$, Lemma \ref{forumlaTNF} asserts $\lambda  = -2 \pi \ T\cdot N_{\mathcal F}$.  Finally Proposition \ref{p: cohomological interpretation of the mass of the harmonic measure} provides $ T  \cdot T_{\mathcal F}  = -{1 \over 2\pi} \int _S \text{vol}_{g_P} \wedge T = -{1 \over 2\pi}$. \end{proof}

\begin{remark} In \cite{Nguyen}, Nguyen computed  the Lyapunov exponent with respect to the metric on $S$. Actually, the result does not depend on the chosen transverse metric (smooth, product or ambient), see Proposition \ref{p: lyapunov product}. \end{remark}

\subsection{Dynamical entropy $h_D$}\label{ss: DE}

As in Equation \eqref{eq: local expression harmonic current}, let $\tau(z,t)$ be leafwise positive harmonic functions defining $T$ in a foliation box. For every transverse section \(\Sigma\) given by \( \{ (z(t), t) \} \), we define the Radon measure 
$$ T_{|\Sigma} :=  \tau (z(t), t ) \nu (dt) . $$ 
Let us recall that $\beta = d_{\FF} \log \tau$, see Equation \eqref{eq: beta}.
Given a holonomy map as in Section \ref{ss: LE}, we have for the Radon-Nikodym derivative:
\begin{equation}\label{eq: Radon Nikodym derivative} 
\frac{D (h_{\gamma_{[0,\Timet]}, \Sigma_0, \Sigma_\Timet})_*^{-1} T_{|\Sigma_\Timet} }{D T_{|\Sigma_0}}(\gamma(0)) = \exp   \int _{\gamma_{[0,\Timet]}}  \beta .  \end{equation} 
The integral of \(\beta\) can also be defined along continuous  paths by homotoping.  

We prove below that $- {1 \over \Timet} \int _{\gamma_{[0,\Timet]}}  \beta$ converges  for \(W_{g_P}^\mu\)-a.e. \(\gamma \in \OO\) to a universal number that we now define. By Lemma \ref{lemmaharnack}, $d d^c_\FF \log \tau$ is bounded with respect to the Poincar\'e metric \(g_P\). Hence the \emph{dynamical entropy} of \(T\)  
\begin{equation}\label{ccc}
h_D  :=  - \int_{S^*} d d_\FF ^c \log \tau  \wedge T  = - \int_{S^*} \Delta_{g_{P}} \log \tau \  T \wedge \text{vol}_{g_P} 
\end{equation}
is well defined. It  was introduced by Kaimanovich \cite{Kaimanovich} and by Frankel \cite{Frankel}. Because \(\tau\) is leafwise harmonic, we also have the formula
 $$h_D  = \int_{S^*} d_\FF \log \tau \wedge d_\FF^c \log \tau \wedge T ,$$ which yields $h_D \geq 0$. 

\begin{proposition} \label{p: integral of beta}
Let $S$ be an algebraic surface endowed with a foliation $\FF$ with hyperbolic singularities and no foliated cycle. Let \( T\) be the unique directed harmonic current satisfying \( \int T\wedge \text{vol}_{g_P} = 1\). Then 
\[ \textrm{for } W_{g_P}^\mu-a.e. \  \gamma \in \OO  \ , \   \lim_{t \to + \infty} \frac{1}{\Timet} \log \frac{D (h_{\gamma_{[0,\Timet]}, \Sigma_0, \Sigma_\Timet})_*^{-1} T_{|\Sigma_\Timet} }{D T_{|\Sigma_0}} (\gamma(0)) = - h_D . \] 
\end{proposition}

\begin{proof} Let us verify that  \( \beta = d_\FF \log \tau  \) satisfies Properties $(i)$ and $(ii)$ of Section \ref{ss: integral of leafwise closed one forms} with \(E= S^*\). Lemma  \ref{lemmaharnack} asserts that the $g_P$-norm of  $\beta$ is bounded. The formula $g_P= e^{-\xi} g_s$ given by Equation \eqref{eq: def phi} implies that $\vert \beta \vert^2_{g_s} = e^{-\xi} \vert \beta \vert^2_{g_P}$. Since $e^{-\xi} \sim (- \log \left(|x|^2+|y|^2 \right) )^{-1}$ is bounded near the singular set, the $g_s$-norm of $\beta$  is also bounded and Property $(i)$ follows. Lemma  \ref{lemmaharnack} gives that $\theta := \Delta_{g_P} \log \tau$ is bounded on $S^*$, hence Property $(ii)$ is satisfied. Finally, Theorem \ref{p: expression of a} applied with $H_\Timet (\gamma) = \int_{\overline \gamma_{[0,\Timet]}} \beta$ yields 
$$  \textrm{for } W_{g_P}^\mu-a.e. \  \gamma \in \OO  \ , \   \lim_{\Timet \to + \infty} \frac{1}{\Timet} \int _{\gamma_{[0,\Timet]}} \beta = A , $$
where $A := \int_{S^*} \theta \, d\mu$. We get $A = - h_D$ from $\mu = T \wedge \text{vol}_{g_P}$ and Equation \eqref{ccc}. We conclude by using Equation \eqref{eq: Radon Nikodym derivative}.
\end{proof}

We complete this section with Proposition \ref{p: cochd}, which will be used for the proof of Kaimanovich inequality, see Section \ref{appendix}. Let us first note that \(\mu \)-a.e. leaf \(L\) has no holonomy. This is a general fact, but in our situation this can be easily seen by noticing that analyticity implies that  there is merely a countable number of leaves with holonomy.  So the desintegration of the harmonic current \(T\) along  \(\mu \)-a.e. leaf \(L\) is a single-valued positive harmonic function \( \tilde \tau \), well defined up to multiplication by a positive constant.

\begin{proposition}\label{p: cochd}
For every $\Timet > 0$, we have 
$$ - h_D = {1 \over \Timet} \int_{S^*} \int_{L_x} \log \left( {{\tilde \tau}(y) \over {\tilde \tau} (x) } \right) p(x,y,\Timet) \text{vol}_{g_P} (dy) d\mu(x) . $$
\end{proposition}

\begin{proof}
We continue the proof of Proposition \ref{p: integral of beta}.
By Theorem \ref{p: expression of a}, $$\forall \Timet > 0 \ , \ A = {1\over \Timet} \int_{S^*} \int_{\Gamma^x}  H_\Timet(\gamma) dW^x_{g_P} (\gamma) \, d\mu(x) .$$
 The desired formula follows from $H_\Timet(\gamma) = \int_{\overline \gamma_{[0,\Timet]}} d_\FF \log {\tilde \tau} = \log \left( {{\tilde \tau} (\gamma(\Timet)) \over {\tilde \tau}(\gamma(0)) } \right)$. \end{proof}

\begin{remark} Proposition \ref{p: integral of beta} and Equation  \eqref{eq: Radon Nikodym derivative}  seem to show that the logarithm of the Radon-Nikodym derivative of a holonomy map along a Brownian trajectory  decreases linearly at the speed $-h_D$. This is not completely rigourous, the problem being that we need to consider an uncountable number of transversals along the Brownian trajectory, each leading to a set of full transversal measure but not for all simultaneously. We will overcome this problem with the discretization process of Section \ref{ss: discretization}. 
\end{remark}

\part{Dimension of transversal measures}

\section{Projection $\pi$ and product metric}\label{ss: discretization}

In this section, we use the fact that the singularities are linearizable to introduce projections $\pi : S^* \to S \setminus B$, where $B := \cup_p B_p$ is the union of the neighborhoods  of singular points $p$ defined in Section \ref{sub:singpoints}. This will be useful in Section \ref{sss: cov} to define a \emph{finite} open covering of $S^*$ by foliation boxes. 

\subsubsection{The projection $\pi$}\label{s: angdom}  
Recall that  $\FF$ is defined on each $B_p$ by a vector field $V = ax \partial_x + by \partial_y$, normalized so that the real part of $V$ is a source ($a$ and $b$ have positive real parts), see Section \ref{sub:singpoints}.

 \begin{definition}
 The projection $\pi$ is the identity mapping on $S \setminus B$ and is defined on $B^*$ as the first exit point in  $\partial B$ of the real part  of $V$. 
 \end{definition}
 
 We will need the expression of $\pi$. Recall that $\Gamma_{(x_0,y_0)} :  u \in A \mapsto (x_0 \exp (a u) , y_0 \exp (b u))$,  given by (\ref{eq: parametrization}), parametrizes the leaves in $B_p$, where $A$ is the angular
 domain $\{ u \in \C \, , \, \Gamma_{x_0, y_0} (u) \in B_p \}$.  Let us fix
$$(x_0,y_0) \in \mathbb S^1 \times \mathbb S^1 \ \textrm{ and } \ u = v+i w \in A . $$
The function  $\tilde v \mapsto \vert \Gamma_{x_0, y_0} (\tilde v + i w)\vert_\infty$ is strictly increasing as $\tilde v$ increases. Let $v_0$ be the unique real number\footnote{Recall that $\partial A = \{ u \in \C \, ,  \, \vert \Gamma_{x_0,y_0}(u) \vert_\infty = 1 \}$.} such that $\vert \Gamma_{x_0, y_0} (v_0 + i w) \vert_\infty =1$. 
We denote 
$$ \Lambda : [0,1] \to A \ , \ \Lambda(\zeta) := (1-\zeta) (v+iw) + \zeta (v_0 + iw) . $$
Now, given $q \in B_p^*$, there exist $(x_0,y_0) \in \mathbb S^1 \times \mathbb S^1$ and $u \in A$ (both unique) such that $q = \Gamma_{x_0, y_0}(u)$. We define 
$$ \Lambda_q :=  \Gamma_{x_0, y_0} \circ \Lambda : [0,1] \to L_q \cap B_p  . $$
The projection $\pi : S^*  \to S \setminus B$ is precisely defined by 
$$ \forall q \in B_p^* \ , \ \pi(q) :=  \Lambda_q(1) \in \partial B_p   . $$
 Now let $g_s$ be a smooth metric on  \( T_\FF \) as defined in Section \ref{smoothmetric}. Recall that $l_g$ stands for the $g$-length of leafwise paths. 

\begin{lemma}\label{lemma: angularproj} $ $
\begin{enumerate}
\item  There exists $C \geq 1$ such that $l_{g_s}(\pi \circ \tau) \leq C l_{g_s}(\tau)$ for every leafwise smooth path $\tau: [0,1] \to B_p^*$. 
\item Let $\epsilon > 0$. For $W^\mu_{g_P}$-a.e. $\gamma \in \Gamma$, if $\gamma(n) \in B_p$, then\footnote{The upper estimate on $l_{g_P}(\Lambda_{\gamma(n)})$ will be used in Section \ref{s: chain}.}
 $d_P(\gamma(n) , \partial B_p) \leq l_{g_P}(\Lambda_{\gamma(n)}) \leq \log (\epsilon n)$ for every $n \geq m_0(\gamma)$. 
\end{enumerate}
\end{lemma}

\begin{proof}
The first Item follows from the fact that $g_s$ is equivalent to the euclidian metric in $A$.
 For the second Item, recall that $\gamma \in \Gamma \to \rho(\gamma(0)) \in \R$ is $W^\mu_{g_P}$-integrable, see Section \ref{ss: harmonicm}. Hence, by Birkhoff ergodic theorem, we have $\rho (\gamma(n)) < \epsilon n$ for  $W^\mu_{g_P}$-a.e. $\gamma \in \Gamma$ and  $n$ large. Lemma \ref{l: inthedomain} then implies  
\begin{equation*}
 d_P^A (\gamma(n) , \partial A) \sim  \log d_{e}^A (\gamma(n) , \partial A )  \leq \log \rho (\gamma(n)) <  \log (\epsilon n) ,
 \end{equation*}
which proves the second Item. 
\end{proof}

\subsubsection{Product metric}\label{s:PM}
 Let $m$ be the metric on $N_\FF$ defined as in Section \ref{normalbundle} by $m = \vert \omega \vert$ near the singular set.  In each $B_p^*$, the foliation $\FF$ is diffeomorphic to the product of the semi-line \((-\infty, 0]\) by a real one dimensional foliation of the \(3\)-sphere $\partial B_p^*$, given by the opposite of the first exit time of $B_p^*$ and by the first exit point in $\partial B_p$. This motivates the following terminology.

 \begin{definition} We define the \emph{product metric} $m_{pr}$ on $N_\FF$ over $S^*$ by 
 $$m_{pr} := m \textrm{ on } S \setminus B \ \ , \ \ m_{pr} := \pi^* (m_{\vert \partial B}) \textrm{ on } B^*, $$
 where $\pi : S^* \to S \setminus B$ is the projection defined in Section \ref{s: angdom}.
 \end{definition}
 
 \begin{proposition}\label{p: lyapunov product}
Let $(S,\FF)$ be a foliated algebraic surface with hyperbolic singularities and no foliated cycle. Let \( T\) be the unique directed harmonic current satisfying \( \int T\wedge \text{vol}_{g_P} = 1\). Let $\lambda $ be the Lyapunov exponent defined in Proposition \ref{d: lyapunov exponent}. Then for \(W_{g_P}^\mu\)-a.e. \(\gamma \in \OO\): 
\[ \lim _{\Timet\rightarrow +\infty} {1 \over \Timet} \log \vert Dh_{\oo_{[0,\Timet]},\Sigma_0,\Sigma_\Timet}(\oo(0)) \vert_{m_{pr}}  = \lambda  .\]
\end{proposition}

\begin{proof}
We write locally $m_{pr} = e^ {\varphi_{pr}}   | dt |$ and define the leafwise \(1\)-form $\eta_{m_{pr}} := d_\FF \varphi_{pr} $.
As in the proof of Proposition \ref{d: lyapunov exponent}, we have  
\[ \forall \gamma \in \Gamma \ , \  \log \vert Dh_{\oo_{[0,\Timet]},\Sigma_0,\Sigma_\Timet}(\oo(0)) \vert_{m_{pr}} = \int_{\oo_{[0,\Timet]}} \eta_{m_{pr}} = H_\Timet ^{\eta_{m_{pr}}} (\oo) . \]
 The form $\eta_{m_{pr}}$ satisfies Property $(i)$, since it is determined by the restriction of $m$ on the compact set $\partial B$, hence is bounded with respect to the leafwise smooth metric $g_s$. By Theorem \ref{p: expression of a}, there exists $\lambda_{pr} \in \R \cup \{ - \infty  \}$ such that 
$$  \textrm{for } W_{g_P}^\mu-a.e. \  \gamma \in \OO  \ , \  \lim _{\Timet\rightarrow +\infty} {1 \over \Timet} \log \vert Dh_{\oo_{[0,\Timet]},\Sigma_0,\Sigma_\Timet}(\oo(0)) \vert_{m_{pr}} = \lambda_{pr}. $$
To conclude it remains to observe that $\lambda_{pr} = \lambda$. This follows from Proposition \ref{d: lyapunov exponent}, from the fact that \(W_{g_P}^\mu\)-a.e. \(\gamma \in \OO\) leaves $B$ for arbitrary large times and from the definition $m_{pr} = m$ on $S\setminus B$. \end{proof} 

\section{Covering, crossings, discretization}\label{s: ccd}

\subsection{A finite covering by foliation boxes} \label{sss: cov} In this section, we assume for simplicity that there is only one singular point, the general case being worked out similarly. Let $B$ be a linearization domain around it, biholomorphic to the bidisc, and let $\partial B$ denote its boundary. Let $(U^+_l , t_l)_{l}$ be a countable family of foliation boxes $U^+_l \simeq \D(2) \times \D(2)$ for $\FF$ which covers $S^*$, together with holomorphic first integral $t_l : U^+_l \to \D(2)$. 

The covering $(U^+_l)_{l}$ can be chosen regular, see \cite[Section 1.2.A]{Candel-Conlon}, which means that if $U^+_l \cap U^+_{l'} \neq \emptyset$, then there exists on this intersection a holomorphic change of coordinates of the form 
\begin{equation}\label{def:change}
 (z , t ) \in U^+_l \mapsto (z_{l ,l'} (z ,t) , t_{l ,l'}(t)) \in U^+_{l'} . 
\end{equation}

Let $U_l^- \subset U_l$  denote the foliation boxes $\subset U^+_l$  isomorphic to $\D({1 \over 2}) \times \D({1 \over 2})$ and $\D \times \D$.  One can assume that $(U_l^-)_l$ covers $S^*$. Let $I$ be a finite subset of $l$'s such that $(U^-_i)_{i \in I}$ covers the compact set $S \setminus B$. Let $$\Sigma_i := \{0 \} \times \D .$$ Let $J \subset I$ be the subset of indices $j \in I$ such that $U_j^-$ intersects $\partial B$. 
Let $$(U_k, t_k)_{k \in K} := (U_i , t_i)_{i \in I \setminus J} \ \cup \ (\pi^{-1}(U_j) , t_j \circ \pi)_{j \in J} ,$$
 where $\pi : S^* \to S \setminus B$ is the leafwise projection map defined in Section \ref{s: angdom}. 
We set the following: the transversal of the foliation box $\pi^{-1}(U_j)$ is $\Sigma_j$ (we can make this choice since $\Sigma_j \subset U_j \subset \pi^{-1}(U_j)$). 

All these yield a finite open covering of $S^*$ by foliation boxes, together with transversals $\Sigma_k$ and first integrals $t_k : U_k \to \D$, where $\D$ is the second coordinate of the foliation box $U_k$.  
 
\subsubsection{Definition of  $\rho_0$} \label{sss: cov2} 
Let us fix a smooth metric \(g_s\) on  \(T_\FF\). By the above construction, there exists $\rho_0 >0$ such that for every $i \in I$ and for every $q \in U^-_i$, the $g_s$-leafwise ball of center $q$ and radius $\rho_0$ is included in $U_i$. In particular, since $(U^-_i)_{i \in I}$ covers $S \setminus B$, we get the following property, used to prove Proposition \ref{p: estimq}: for every $q \in S \setminus B$, there exists $i \in I$ such that the $g_s$-leafwise ball of center $q$ and radius $\rho_0$ is included in $U_i$.

\subsubsection{Definition of $\delta_0$} \label{sss: box} If $U_k \cap U_{k'} \neq \emptyset$, we get from Equation (\ref{def:change}) with $(l, l') = (k,k')$ a map $$t_{k,k'} : t_k(U^+_k \cap U^+_{k'}) \subset \D(2) \to t_{k'}(U^+_k \cap U^+_{k'}) \subset \D(2) . $$
Since $U_l \Subset U_l^+$, there exists $\delta_{k,k'} > 0$ such that for every  $t \in t_k (U_k  \cap U_{k'})$, 
$$\D_t(\delta_{k,k'}) \subset t_k (U^+_k \cap U^+_{k'}) \textrm{ and } t_{k,k'} (\D_t(\delta_{k,k'}))  \subset t_{k'}(U^+_k \cap U^+_{k'}) . $$
 Let $\delta_0$ be the minimum of those numbers when $k,k'$ vary in $K$. In particular, if $U_k \cap U_{k'} \neq \emptyset$ then for every $t \in t_k (U_k  \cap U_{k'})$, 
 \begin{equation}\label{htkk}
   h_{t, k,k'} := t_{k,k'} :  \D_t(\delta_0) \subset t_k (U^+_k) \to t_{k'}(U^+_{k'})  \subset \D(2)
 \end{equation}
 is well defined. Cauchy's inequality then yields $\theta >0$ such that
 \begin{equation}\label{cauch}
  \forall t \in t_k (U_k  \cap U_{k'}) \ , \ \forall u \in \D_t (\delta_0 /2) \ , \ \vert h'_{t, k,k'} (u) \vert \leq e^\theta . 
  \end{equation} 

\begin{remark}\label{rk: idhol} If $(U_k,t_k), (U_{k'} , t_{k'})$ are given by $(\pi^{-1}(U_j), t_j \circ \pi) , (U_j , t_j)$ for some $j \in J \subset I$, where $I, J$ are defined in Section \ref{sss: cov}, then $$t_{k,k'} : t_j\circ \pi ( \pi^{-1}(U^+_j) \cap U^+_j) \subset \D(2) \to t_j (\pi^{-1}(U^+_j) \cap U^+_j) \subset \D(2)  $$
is the identity mapping of $t_j(U^+_j)$. We use it in the proof of Lemma \ref{l: cl1}. \end{remark}

\subsection{Integrability of the number of crossings}\label{sss: cross} 

\begin{definition}\label{defq}
Let $\tau : [0,1] \to S^*$ be a leafwise continuous path such that $\tau[0,1]$ is not contained in some $U_k$. Let $k_- , k_+$ such that $\tau(0) \in U_{k_-}$ and  $\tau(1) \in U_{k_+}$. We define $q(\tau)$ as the smallest integer $q \geq 1$ such that there exist $(\tilde k_i)_{0 \leq i \leq q} \in K^{q+1}$ and a  subdivision 
$$0 =: \Timet_0 < \Timet_1 < \ldots < \Timet_q < \Timet_{q+1} := 1 $$  such that
$$ \tau [\Timet_i,\Timet_{i+1}] \subset U_{\tilde k_i}  \ \ , \ \  \forall i = 0,\ldots, q  , $$
with $\tilde k_0 = k_- $ and $\tilde k_q = k_+$. 
\end{definition}

\begin{definition}\label{defQ}
For every $\gamma \in \Gamma$, let $Q(\gamma)$ be the minimum of $q(\tau)$, where $\tau : [0,1] \to S^*$ runs over the  leafwise continuous paths homotopic to $\gamma_{[0,1]}$ with fixed extremities. Observe that   $Q(\gamma)$ only depends on $\gamma_{[0,1]}$.
\end{definition}

 We recall that $D_1^{g_s}$ is defined in Equation \eqref{eq: cocycle smooth distance}. 

\begin{proposition}\label{p: estimq} There exists $\zeta \geq 1$ such that $Q  \leq \zeta D_1^{g_s}$ on $ \Gamma$.
In particular, according to Proposition \ref{p: integrability}, the function $Q$ is $W^\mu_{g_P}$-integrable. 
\end{proposition}
Before proving Proposition \ref{p: estimq}, let us observe that Birkhoff ergodic Theorem implies the following corollary, where 
\begin{equation}\label{mplus} M_- := {1\over 2} \int Q \,  dW^\mu_{g_P} \ \ , \ \ M_+ := 2 \int Q \, dW^\mu_{g_P} .
\end{equation}

\begin{corollary}\label{c: estimq}
For $W^\mu_{g_P}$-almost every $\gamma \in \Gamma$, for every $n \geq n_0(\gamma)$,
\begin{equation*}
M_- \leq {1 \over n} \sum_{j=0}^{n-1} Q(\sigma^j_1 \gamma) \leq M_+ \ \ \textrm{ and } \ Q (\sigma^n_1 \gamma) \leq n \epsilon .   
\end{equation*}
\end{corollary}

\begin{proof} {\it (of Proposition \ref{p: estimq}).}
 Let $\gamma \in \Gamma$ and let $\tau : [0,1] \to S^*$ be a leafwise smooth path homotopic to $\gamma_{[0,1]}$ such that $l_{g_s}(\tau) = D^{g_s}_1 (\gamma)$. Let us introduce a modification $\bar \tau$ of $\tau$, which will be a leafwise and piecewise smooth path homotopic to $\gamma_{[0,1]}$ with fixed extremities. For simplicity we assume that there is only one singular point, let $\pi : S^* \to S \setminus B$ be the projection map defined in Section \ref{s: angdom}. If $\tau ^{-1}(B)$ is empty, then we set $\bar \tau :=  \tau$. If not, 
\begin{equation}\label{eq: ja} \tau^{-1}(B) = [0,\Times[  \ \cup \ \left(\cup_a J_a \right)   \ \cup \  ]\Times',1] ,
\end{equation}
where $[0,\Times[$, $]\Times',1]$ may be empty, and the $J_a$'s are pairwise disjoint open intervals of $[\Times,\Times']$. If $[0,\Times[$, $]\Times',1]$ were empty, then we define 
$$\bar \tau  := \pi \circ  \tau  \textrm{ on}   \cup_a  J_a \ \ , \ \ \bar \tau := \tau \textrm{ elsewhere.}$$
If $[0,\Times[$ were not empty (we adopt a similar definition on $]\Times',1]$ if necessary), we define $\bar \tau$ on $[0,\Times[$ by  
$$ \forall \Timet \in [0,\Times/2] \ , \   \bar \tau (\Timet) := (1- {2\Timet \over \Times}) \, \tau(0) + {2\Timet \over \Times}   \, \pi(\tau (0)) , $$
$$ \forall \Timet \in [\Times/2,\Times[ \ , \ \bar \tau (\Timet) := \pi ( \tau(2\Timet-\Times)) . $$
The first part links $\tau(0)$ and $\pi(\tau(0))$ by a line in the angular domain $A$, the second part is drawn on $\partial B$.
 
Now let us find $\zeta$ independent from $\gamma$ such that $q(\bar \tau) \leq \zeta D_1^{g_s}(\gamma)$, this will prove Proposition \ref{p: estimq} since $Q(\gamma) \leq q(\bar \tau)$. Let us begin by restricting to an interval $E$ of $[0,1] \setminus \tau^{-1}(B)$. Since $\bar \tau_{\vert E}  : E \to S \setminus B$ and $\bar \tau_{\vert E} =  \tau_{\vert E}$, we get $$ \rho_0 \, q(\bar \tau_{\vert E})  \leq  l_{g_s}(\bar \tau_{\vert E}) =  l_{g_s}(\tau_{\vert E}) , $$
 where $\rho_0$ is defined in Section \ref{sss: cov2}. Now let $J_a \subset \tau^{-1}(B)$ be one of the intervals occuring in Equation (\ref{eq: ja}). Since $\bar \tau_{\vert J_a} : J_a \to \partial B \subset S \setminus B$, we get as before $\rho_0 \, q(\bar \tau_{\vert J_a})  \leq l_{g_s}(\bar \tau_{\vert J_a})$.
 Lemma \ref{lemma: angularproj} implies that $l_{g_s}(\bar \tau_{\vert J_a}) \leq C  l_{g_s}(\tau_{\vert J_a})$ for some constant $C \geq 1$, hence
 $$\rho_0 \, q(\bar \tau_{\vert J_a})  \leq C \,  l_{g_s}(\tau_{\vert J_a}).$$
Finally, if $[0,\Times[$ is not empty, a similar estimate occurs for $q(\bar \tau_{[0,\Times[})$ since $\bar \tau ([0,\Times/2])$ is contained in some $\pi^{-1}U_j$ (containing $\tau(0)$) and $\bar \tau ([\Times/2,\Times]) \subset \partial B$. To conclude, with $\zeta := C/\rho_0$, we get $q(\bar \tau)  \leq \zeta l_{g_s}(\tau) = \zeta D^{g_s}_1 (\gamma)$. \end{proof}

\subsection{Discretization of the dynamical system \( ( \OO, \sigma, W_{g_{P}}^\mu) \)}\label{ss: diext} Let $(U_k)_{k \in K} \Subset (U^+_k)_{k \in K}$ be the two finite coverings of $S^*$  introduced in Section \ref{sss: cov}. Given $\gamma \in \Gamma$, we shall need fixed foliation boxes around the $\gamma(n)$'s. For that purpose we introduce a fibered dynamical system.

Let $\psi_k : U^+_k \to \R^+$ be smooth functions such that $\sum_{k \in K} \psi_k = 1$ on $S^*$. 
Let us define the family of probability measures $(\kappa_q)_{q \in S^*}$ on $K$:
\begin{equation}\label{eq:psipsi}
\forall q \in S^*  \ \ , \ \ \kappa_q := \sum_{k \in K} \psi_k( \pi (q)  ) \, \delta_k 
\end{equation}
where $\delta_k$ is the Dirac mass at $k$ and $\pi : S^* \to S \setminus B$ is the projection defined in Section \ref{sss: cov}. We define the probability measure $W_K$ on $\Gamma \times K^\N$ by
$$ W_K( \Gamma' , L ) := \int_{\gamma \in \Gamma'} (\otimes_{n \in \N} \, \kappa_{\gamma(n)})(L) \, dW^\mu_{g_P} (\gamma)  $$
for every Borel sets $\Gamma' \subset \Gamma$ and $L \subset K^\N$. Let us  introduce  
$$ F  : \Gamma  \times K^\N  \to  ( S^* \times K)^\N \ , \  F(\gamma , (k_n)_{n \in \N}) = ( \gamma(n), k_n)_{n \in\N}  \ , \ W^1_K  := F_* (W_K) .$$ 
Observe that if $\pi_0$ denotes the projection $(p_n , k_n)_n \mapsto p_0$ then $(\pi_0)_* W^1_K= \mu$. 
Since the measure $W^\mu_{g_P}$ on $\Gamma$ is $\sigma_1$-invariant,  the measure $W^1_K$ on $( S^* \times K)^\N$ is invariant by the left shift $\sigma_K$, and we get a dynamical system $$( ( S^* \times K)^\N , \sigma_K , W^1_K) .$$
Further, we shall need leafwise continuous path parametrized by $\R$: this appears in Section \ref{proofthmB}, where the target $\gamma(n)$ replaces the initial point $\gamma(0)$. We simply use the natural extension $(\widehat \Gamma , \widehat \sigma,  \widehat W)$ of $(\Gamma , \sigma_1 , W^\mu_{g_P})$, see \cite[Section 10.4]{CFS}. The elements of $\widehat \Gamma$ are sequences $\widehat \gamma = (\ldots , \gamma_{-1} , \gamma_0 , \gamma_1 , \ldots )$ of leafwise continuous paths  satisfying $\sigma_1 (\gamma_n) = \gamma_{n+1}$ for every $n \in \Z$. The map $\widehat \sigma$ is the left shift and $\widehat W$ is the unique $\widehat \sigma$-invariant measure on $\widehat \Gamma$ satisfying $$\widehat W (s^{-1} \Gamma') =  W^\mu_{g_P}(\Gamma')$$ for every Borel set $\Gamma' \subset \Gamma$, where $ s : \widehat \Gamma \to \Gamma$ is the projection map $\widehat \gamma \mapsto \gamma_0$.
 
As before, we define the probability measure $\widehat W_K$ on $\widehat \Gamma \times K^\N$ by
$$ \widehat W_K( \widehat \Gamma' , L ) := \int_{\widehat \gamma \in \widehat \Gamma'} (\otimes_{n \in \N} \, \kappa_{\gamma_0(n)})(L) \, d\widehat W (\widehat \gamma)  $$
for every Borel sets $\widehat \Gamma' \subset \widehat \Gamma$ and $L \subset K^\N$. Let also 
$$ \widehat F  : \widehat \Gamma  \times K^\N  \to  ( S^* \times K)^\Z \ , \  F(\widehat \gamma , (k_n)_{n \in \Z}) = ( \gamma_n(0), k_n)_{n \in\Z} .  $$
The measure $\widehat W^1_K  := \widehat F_* (\widehat W_K)$ on $( S^* \times K)^\Z$ is invariant by the left shift $\widehat \sigma_K$. Here we get a dynamical system $( ( S^* \times K)^\Z , \widehat \sigma_K , \widehat W^1_K)$. 

\section{Domain of definition and contraction of holonomy maps}

We study holonomy maps defined in Section \ref{ss: holonomy}.
We shortly denote
 $$h_{\gamma_{[0,n]} , k_0, k_n} := h_{\gamma_{[0,n]} , t_{k_0}, t_{k_n}}  , $$
  where $\gamma(0) \in U_{k_0}$, $\gamma(n) \in U_{k_n}$, and the first integral $t_k : U_k \to \D$ involved in Section \ref{ss: holonomy} is equal to $t_i$ or to $t_j \circ \pi$. By Proposition \ref{p: lyapunov product}, 
  $$  \forall \gamma \in \Gamma \ , \  W_{g_P}^\mu-a.e. \ , \  \lim _{n \rightarrow +\infty} {1 \over n} \log \vert Dh_{\gamma_{[0,n]} , k_0, k_n}(\gamma(0)) \vert_{m_{pr}}   = \lambda. $$
We now want to prove that $h_{\gamma_{[0,n]} , k_0, k_n}$ exists for every $n$ on a domain which does not depend on $n$. It will be crucial that the number $Q$ of crossed foliations boxes is sublinear, see Equation \eqref{choicen0} below.

\begin{remark} The product metric $m_{pr}$ on transversals is defined using the compact subset $S \setminus B$, see Section \ref{s:PM}. We shall replace it with the classical modulus, up to multiplicative constants that we omit. 
\end{remark}

\begin{proposition}\label{prop:holodist}
For $W^1_K$-almost every $(p_n , k_n)_{n \in \N} \in (S^* \times K)^\N$, there exist 
$\gamma \in \Gamma$, $\delta_\gamma > 0$ and $C_\gamma \geq 1$ such that for every $n \geq 0$, $\gamma(n)=p_n$ and
\begin{enumerate}
\item $h_{\gamma_{[0,n]},k_0,k_n}$ is well defined on $\D_{t_{k_0}(p_0)}(\delta_\gamma)$,
 \item $\vert h'_{\gamma_{[0,n]},k_0,k_n}  \vert \leq C_\gamma e^{n(\lambda + 3 \epsilon)}$ on  $\D_{t_{k_0}(p_0)}(\delta_\gamma)$.
\end{enumerate} 
\end{proposition}

The remainder of this Section consists in proving Proposition \ref{prop:holodist}.
Let $\gamma \in \Gamma$ be such that $\gamma(n)=p_n$. 
Corollary \ref{c: estimq} provides $n_0(\gamma)$ such that 
\begin{equation}\label{choicen0}
\forall n \geq n_0 \ , \ Q (\sigma_1^n(\gamma)) \leq n \epsilon.
\end{equation}
 By Proposition \ref{p: lyapunov product}, we can increase $n_0$ so that
\begin{equation}\label{choicen1}
\forall n \geq n_0 \ , \  \vert h'_{\gamma_{[0,n]},k_0,k_n} (t_{k_0}(p_0)) \vert \leq e^{n(\lambda + \epsilon)} .
\end{equation}
Moreover, since $\lambda < 0$, one can also assume that
\begin{equation}\label{choicen2}
\forall n \geq n_0 \ , \  e^{n(\lambda + (3+ \theta)\epsilon)} \leq \min \{ \delta_0 /2 , \delta_0 r_\epsilon \} ,
\end{equation}
 where $\delta_0$ and $\theta$ are defined in Section \ref{sss: box}, and $r_\epsilon$ is given as follows:
 
{\it Koebe Distortion Theorem: there exists $r_\epsilon <1$ such that for every  holomorphic injective function $f : \D(R) \to \C$,} \begin{equation}\label{koebedist} \forall t \in \D(R r_\epsilon) \ , \ e^{-\epsilon} \vert f'(0) \vert \leq \vert f'(t) \vert \leq e^{\epsilon} \vert f'(0) \vert .
\end{equation}
 Let us define $$h_n := h_{\gamma_{[0,n]},k_0,k_n}  $$
 and let $\tilde \delta <1$ be  such that $h_{n_0}$ is well defined (and thus injective as an holonomy map) on $\D_{t_{k_0}(p_0)}(\tilde \delta)$ and takes its values in $\D_{t_{k_{n_0}}(p_{n_0})}(\delta_0 r_\epsilon)$. Let  $\delta_\gamma := \tilde \delta r_\epsilon$, simply denoted $\delta$ for the remainder of the proof, and $D_\delta := \D_{t_{k_0}(p_0)}(\delta)$. 
 Equations (\ref{koebedist}) and (\ref{choicen1}) successively yield
  \begin{equation*}\label{EQ1}
   \vert h'_{n_0} \vert \leq e^{\epsilon} \vert h'_{n_0} (t_{k_0}(p_0)) \vert  \leq e^{n_0(\lambda + 2\epsilon)} \textrm{ on } D_\delta. 
    \end{equation*}
 In particular, taking into account $\delta <1$ and Equation (\ref{choicen2}), we obtain 
 \begin{equation*}\label{EQ2}
 h_{n_0}(D_\delta) \subset \D_{ t_{k_{n_0}}(p_{n_0}) }(e^{n_0(\lambda+2\epsilon)}) \subset \D_{t_{k_{n_0}}(p_{n_0})}(\delta_0 r_\epsilon) .
 \end{equation*}
To establish Proposition \ref{prop:holodist}, it suffices to prove $(\varphi_n)$ and $(\alpha_n)$ for $n \geq n_0$:
$$ (\varphi_n) : h_n \textrm { is well defined and injective on } D_\delta ,$$
$$ (\alpha_n) : \vert h'_n \vert \leq  \vert h'_{n} (t_{k_0}(p_0)) \vert e^{2n\epsilon} \leq e^{n(\lambda + 3\epsilon)} \textrm{ on } D_\delta . $$
Both are obviously satisfied for $n=n_0$. Assume that $(\varphi_n)$, $(\alpha_n)$ are true and let us prove $(\varphi_{n+1})$, $(\alpha_{n+1})$. Recall that $\gamma^n$ is the restriction to $[0,1]$ of the $n$-shifted path $\sigma^n_1 \gamma$, it satisfies $\gamma^n(0) \in U_{k_n}$, $\gamma^n(1) \in U_{k_{n+1}}$. 

\begin{lemma}\label{needto}
Let $q:= Q(\gamma^n)$ and let $\tau$ be homotopic to $\gamma^n$ with fixed extremities such that there exist
 $(\tilde k_i)_{0 \leq i \leq q} \in K^{q+1}$ and $0 =: \Timet_0 < \Timet_1 < \ldots < \Timet_q < \Timet_{q+1} := 1$ such that (see Definitions \ref{defq} and \ref{defQ})
$$ \forall i = 0,\ldots, q  \ , \ \tau [\Timet_i,\Timet_{i+1}] \subset U_{\tilde k_i}  \ , \textrm{ with } \tilde k_0 = k_n \ , \ \tilde k_q = k_{n+1} .$$  Observe that by setting 
$$t_0 := t_{\tilde k_0} (\tau(0)) \ , \ t_i := t_{\tilde k_i} ( \tau (\Timet_i) ) \in t_{\tilde k_i} (U_{\tilde k_{i-1}} \cap U_{\tilde k_i})  $$
we get the following decomposition, where $h_{t,k,k'}$ is defined in Equation \eqref{htkk},
$$ h_{\gamma^n, \Sigma_{k_n} , \Sigma_{k_{n+1}} } = h_{t_{q-1}, \tilde k_{q-1}, \tilde k_q} \circ \ldots \circ h_{t_0, \tilde k_0, \tilde k_1} . $$
Then, for every $j = 0 , \ldots , q$, we have the following set $Z_j$ of three assertions, where $g_0 := Id_{\vert T_{k_n}}$ and $g_j := h_{t_{j-1}, \tilde k_{j-1}, \tilde k_j} : \D_{t_{j-1}}(\delta_0) \to \D(2)$: 
 $$ (\phi_j) : g_j \circ \ldots \circ g_0 \circ h_n \textrm { is well defined and injective on } D_\delta , $$ 
 $$ (\tau_j) : \vert (g_j \circ \ldots \circ g_0 \circ h_n)' \vert \leq e^{j\theta} e^{n(\lambda + 3 \epsilon)} \leq e^{n(\lambda + (3+\theta)\epsilon)} \textrm{ on } D_\delta ,  $$
 $$ (\sigma_j) : \vert (g_j \circ \ldots \circ g_0 \circ h_n)' \vert \leq  \vert (g_j \circ \ldots \circ g_0 \circ h_n)'(t_{k_0}(p_0)) \vert e^{(n+j)\epsilon} \textrm{ on } D_\delta . $$
  \end{lemma}

Let us verify that Lemma \ref{needto} (namely $Z_q$) implies $(\varphi_{n+1})$ and $(\alpha_{n+1})$. Let us observe that 
 $$ h_{n+1} =  h_{\gamma^n, \Sigma_{k_n} , \Sigma_{k_{n+1}} } \circ h_n = g_q  \circ \ldots \circ g_0 \circ h_n . $$
 Assertion $(\varphi_{n+1})$ is then  provided by $(\phi_q)$. The first inequality of $(\alpha_{n+1})$ then follows by using $(\sigma_q)$ and $q\epsilon \leq n \epsilon \cdot \epsilon \leq n \epsilon$ given by Equation (\ref{choicen0}):
$$ \vert h'_{n+1} \vert \leq \vert h'_{n+1}(t_{k_0}(p_0)) \vert e^{(n+q)\epsilon} \leq \vert h'_{n+1}(t_{k_0}(p_0)) \vert e^{2n\epsilon} \textrm{ on } D_\delta  .$$
Finally the second inequality of $(\alpha_{n+1})$ is a consequence of Equation (\ref{choicen1}).  This completes the proof of Proposition \ref{prop:holodist}.

\begin{proof} (of Lemma \ref{needto})
We have to show $Z_j$ for $j = 0 , \ldots , q$. For $Z_0$, observe that $(\phi_0)$ follows from $(\varphi_n)$ and that $(\tau_0)$, $(\sigma_0)$ follow from $(\alpha_n)$. Assume that $Z_j$ is true for some $j = 0 , \ldots, q-1$. By $(\phi_j)$, $(\tau_j)$ and Equation (\ref{choicen2}), 
$$g_j \circ \ldots \circ g_0 \circ h_n(D_\delta) \subset \D_{t_j}( e^{n(\lambda + (3+\theta)\epsilon)}) \subset \D_{t_j}( \min \{ \delta_0 /2 , \delta_0 r_\epsilon \}) , $$  
hence we can compose by $g_{j+1}$ (see Equation \eqref{htkk}), which proves $(\phi_{j+1})$. To get the first inequality of $(\tau_{j+1})$, we use Cauchy's estimates (Equation (\ref{cauch})) for $g_{j+1}$ and the assertion $(\tau_j)$, they indeed yield on $D_\delta$:
$$ \vert (g_{j+1} \circ \ldots \circ g_0 \circ h_n)' \vert \leq \vert g'_{j+1} \vert \vert (g_j \circ \ldots \circ g_0 \circ h_n)' \vert \leq e^{\theta} e^{j\theta} e^{n(\lambda + 3 \epsilon)} .$$
The second inequality  of $(\tau_{j+1})$ follows from $(j+1)\theta \leq q\theta \leq n\theta \epsilon$. It remains to prove $(\sigma_{j+1})$. We have $g_j \circ \ldots \circ g_0 \circ h_n(D_\delta) \subset  \D_{t_j}(\delta_0 r_\epsilon)$
and $g_{j+1}$ is holomorphic and injective on $\D_{t_j}(\delta_0)$. Koebe distortion (Equation (\ref{koebedist})) then  implies for every $t \in D_\delta$, 
$$ \vert (g_{j+1} \circ \ldots \circ g_0 \circ h_n)'(t) \vert \leq \vert g'_{j+1}(t_j) \vert e^\epsilon \vert (g_j \circ \ldots \circ g_0 \circ h_n)' (t) \vert  , $$
which implies $(\sigma_{j+1})$ by using $(\sigma_{j})$, completing the proof of Lemma \ref{needto}. \end{proof}

 \section{Dynamical entropy $h_D$ and transversal measures}

Recall that for every $k \in K$, we fixed a transversal $\Sigma_k$ in $U_k$, see Section \ref{sss: cov}. Let $T_{\vert \Sigma_k}$ be the restriction of the harmonic current to $\Sigma_k$ as defined in Section \ref{ss: DE}. Let $$\nu_k := (t_k)_* T_{\vert \Sigma_k}$$ be its pushforward by the first integral $t_k : U_k \to \D$. 
 The Lyapunov exponent $\lambda$, the dynamical entropy $h_D$ and the constants $M_+, M_-$ were defined in Sections \ref{ss: LE}, \ref{ss: DE} and Equation \eqref{mplus}.  

\begin{proposition}\label{prop:expentropy}
For $W^1_K$-almost every  $w = ( p_n ,  k_n)_{n \in \N} \in (S^* \times K)^\N$ and for every $n \geq n_1(w)$, 
\begin{enumerate}
\item[1.] $\nu_{k_n} ( \D_{t_{k_n}  (p_n)} ( e^{n (\lambda - \epsilon(1+M_+)) } ) ) \leq e^{-n (h_D - 5 \epsilon) }$, 
\item[2.] $\nu_{k_n} ( \D_{t_{k_n}  (p_n)} ( e^{n (\lambda + \epsilon(1+M_-)) }  ) ) \geq  e^{-n(h_D + 5 \epsilon) }$.
\end{enumerate}
 \end{proposition}

\begin{remark}\label{rk: undiscr}
By Proposition \ref{prop:expentropy} and the definition of $W^1_K$ (Section \ref{ss: diext}), for $\mu$-a.e. $x \in S^*$, for $W^x_{g_P}$-a.e. $\gamma \in \Gamma^x$ and for every $n \geq 0$, we have $\nu_k ( \D_{t_k  (\gamma(n))} ( e^{n (\lambda + \epsilon(1+M_-)) }  ) ) \geq  e^{-n(h_D + 5 \epsilon) }$ whenever $\psi_k(\pi(\gamma(n))) > 0$, where $\psi_k$ is defined in Equation \eqref{eq:psipsi}. We shall use it to prove Proposition \ref{p: separated}.
\end{remark}

\subsection{Radon-Nikodym derivative} \label{RN} 

\subsubsection{The distortion function $\eta$} 
Let $k,k'$ be such that $U_k \cap U_{k'} \neq \emptyset$. Let $t \in t_k(U_k \cap U_{k'})$ and let us set as in Equation (\ref{htkk}):
$$ h := h_{t, k,k'} :  \D_t(\delta_0) \subset t_k (U^+_k) \to t_{k'}(U^+_{k'})  .$$
For every $\delta < \delta_0$, we define $\eta(t,k,k', \delta)$ as follows:
\begin{equation}\label{defeta}
 {\nu_{k'} (h(\D_t (\delta))) \over \nu_k (\D_t (\delta) )} = e^{\eta(t , k,k',  \delta)} {D(h^{-1})_* \nu_{k'}  \over D\nu_k}  (t)  .
 \end{equation}
 Since $\nu_{k'} (h(\D_t (\delta)))$ coincides with $(h^{-1})_* \nu_{k'} (\D_t (\delta))$, by the classical properties of the Radon-Nikodym derivative, there exists a Borel set $R_{k,k'} \subset t_k(U_k \cap U_{k'})$ of full $\nu_k$-measure such that $$ \forall t \in R_{k,k'} \ , \  \lim_{\delta \to 0} \eta(t,k,k', \delta) = 0 . $$
\begin{lemma}\label{lem: Harnacktrans2}
For every $t \in R_{k,k'}$ and  $\delta < \delta_0$, we have $\vert \eta(t , k,k',  \delta) \vert \leq 2c$. 
 \end{lemma}
 
 \begin{proof}
 Let $t \in R_{k,k'}$.
 Recall that the density of $T_{\vert \Sigma_k}$ is given by the leafwise positive harmonic functions $\tau$ (see Section \ref{ss: DE}) and that $\nu_k = (t_k)_* T_{\vert \Sigma_k}$. By construction, the hyperbolic length between the transversals $\Sigma_k$, $\Sigma_{k'}$ is bounded, see Section  \ref{sss: cov}.
  Harnack inequality (\ref{eq: harnack}) then implies that there exists $c>0$ (not depending on $t, k,k'$) such that for every  $B \subset \D_t(\delta_0)$:
$$  e^{-c} \leq \frac{ \nu_{k'} (h(B)) }{\nu_k (B) } \leq e^c . $$
In particular the derivative ${D(h^{-1})_* \nu_{k'} / D\nu_k}$ satisfies the same estimates at $t$. The conclusion of the Lemma follows from Equation \eqref{defeta}.
 \end{proof}

Let $t' := h(t)$. Koebe Theorem (Equation \eqref{koebedist}) allows to approximate $h(\D_t (\delta))$ by $ \D_{t'} ( \vert h' (t) \vert \delta)$ as follows. We need these inequalities to relate the Lyapunov exponent $\lambda$ with the dynamical entropy $h_D$ in Section \ref{s: chain}.
 
\begin{lemma}\label{eq:koebeRN}
There exists $\delta_0' < \delta_0$ such that for every $\delta < \delta_0'$:
$$  {\nu_{k'} ( \D_{t'} (\vert h' (t) \vert \delta e^{-\epsilon}) \over \nu_k (\D_t (\delta) )} \leq e^{\eta(t , k,k',  \delta)} 
{D(h^{-1})_* \nu_{k'}  \over D\nu_k}  (t)  \leq {\nu_{k'} ( \D_{t'} (\vert h' (t) \vert \delta e^{\epsilon}) \over \nu_k (\D_t (\delta) )} . $$
\end{lemma}

\begin{proof}
Koebe Theorem yields $\delta_0' < \delta_0$ such that 
\begin{equation}\label{eq:koebe}
\forall \delta < \delta_0' \ , \ \D_{t'} ( \vert h' (t) \vert \delta e^{-\epsilon}  ) \subset h(\D_t (\delta)) \subset \D_{t'} ( \vert h'(t)  \vert \delta e^{\epsilon}  ) . 
\end{equation}
The conclusion follows from Equation (\ref{defeta}).
\end{proof}

\subsubsection{The limit $\eta \to 0$ holds for generic leafwise paths}
 
\begin{lemma}\label{l: controleW}
Let $\Gamma_0$ be the set of leafwise continuous paths $\gamma \in \Gamma$ such that there exist $\Timet \geq 0$, $k , k' \in K$ satisfying $\gamma(\Timet) \in U_k \cap U_{k'}$ and $t_k(\gamma(\Timet)) \notin R_{k, k'}$. Then $\Gamma_0$ has zero $W^\mu_{g_P}$-measure.   
\end{lemma}

\begin{proof}
For every $k,k' \in K$ such that $U_k \cap U_{k'} \neq \emptyset$, we set 
 $$H_{k,k'} := \{ \gamma \in \Gamma \ , \ \gamma(0) \in U_k \cap U_{k'} \ , \ t_k(\gamma(0)) \notin R_{k, k'} \} .$$ 
According to the expressions of $\mu$ and $T$ (Equations  \eqref{eq: harmonic measure}, \eqref{eq: local expression harmonic current}) and Harnack inequality \eqref{eq: harnack}, there exists $c_k > 0$ such that  $ \mu \leq c_k \, \nu_k \circ t_k$ on $U_k$.  
 Since $R_{k,k'}$ has full $\nu_k$-measure in $t_k(U_k \cap U_{k'})$ and since $W^\mu_{g_P} = \int W^x_{g_P} d\mu(x)$, we get $W^\mu_{g_P}(H_{k,k'}) = 0$. If $H := \cup_{k,k'} H_{k,k'}$, then $\tilde H := \cup_{\Timeq \in \Q^+} \sigma_\Timeq^{-1} H$ has zero $W^\mu_{g_P}$-measure, since $W^\mu_{g_P}$ is $\sigma_\Timeq$-invariant. We finish by observing that $\Gamma_0 \subset \tilde H$. Indeed, if $\gamma(\Timet) \in U_k \cap U_{k'}$ and $t_k(\gamma(\Timet)) \notin R_{k, k'}$, then by continuity of $\gamma$ there exists $\Timeq \in \Q^+$ such that $\gamma(\Timeq) \in U_k \cap U_{k'}$ and $t_k(\gamma(\Timeq)) =  t_k(\gamma(\Timet)) \notin R_{k, k'}$.
\end{proof}

\subsubsection{Control of $\eta$ along leafwise paths}
For every $\gamma \in \Gamma$ and $r < \delta_0$, 
$$\eta_r (\gamma) := \sup_{q(\tau) = Q(\gamma)}  \sum_{i=0}^{Q(\gamma)} \, \sup_{\delta < r }\vert \eta(t_i , \tilde k_i , \tilde k_{i+1} , \delta ) \vert ,$$
the supremum being taken on smooth leafwise paths $\tau : [0,1] \to S^*$ homotopic to $\gamma_{[0,1]}$ with fixed extremities\footnote{$q(\tau), Q(\gamma)$ are defined in Section \ref{sss: cross}. In particular $t_i$, $\tilde k_i$, $\tilde k_{i+1}$ are related to $\tau$.} such that $q(\tau) = Q(\gamma)$. Note that, as the function $Q$, the function $\eta_r$ only depends on the restriction of $\gamma$ to $[0,1]$. The following result will be crucial.

\begin{proposition}\label{l: lebesg}
There exists $0 < r_0 < \delta_0'$ such that for $W^\mu_{g_P}$-a.e. $\gamma \in \Gamma$, 
$$\forall n \geq n_1(\gamma) \ \ , \ \ -n \epsilon \leq \sum_{j=0}^{n-1} \eta_{r_0} (\sigma_1^j(\gamma)) \leq n \epsilon . $$
\end{proposition}

\begin{proof}
By Lemma \ref{lem: Harnacktrans2}, we get 
$$ \forall \gamma \in \Gamma \ , \ \forall r < \delta_0' \ , \ \eta_r(\gamma)  \leq 2c \,  Q(\gamma) .$$ 
Since $\lim_{r \to 0} \eta_r(\gamma) = 0$ for $W^\mu_{g_P}$-almost every $\gamma$ (by Lemma \ref{l: controleW}) and since $Q \in L^1(W^\mu_{g_P})$ (by Proposition \ref{p: estimq}), Lebesgue dominated convergence yields $r_0$ such that $-\epsilon /2 < \int \eta_{r_0}(\gamma) dW^\mu_{g_P}(\gamma) < \epsilon/2$. We conclude by applying Birkhoff ergodic theorem. 
\end{proof}

\subsection{Exponential decay of transversal measures}\label{s: chain}

This section is devoted to the proof of Proposition \ref{prop:expentropy}.
Let us fix $( p_n , k_n)_{n \in \N} \in (S^* \times K)^\N$ and let $\gamma \in \Gamma$ provided by Proposition \ref{prop:holodist}, we have $\gamma(0)=p_0$. Let $q_j := Q(\sigma_1^j(\gamma))$. We decompose $h_{\gamma_{[0,n]}}$ as $q_0 + \ldots + q_{n-1}$ holonomy maps as in the proof of Proposition \ref{prop:holodist}. Taking same notations, we set for every $0 \leq j \leq n-1$ and $1 \leq i \leq q_j$:
 $$h_i^j := h_{t^j_{i-1}, \tilde k^j_{i-1}, \tilde k^j_{i}} \circ \ldots \circ h_{t_0^j, \tilde k^j_0, \tilde k^j_1} \circ h_{\gamma_{[0,j]}} .$$ 
 Recalling that $\delta'_0$ is the constant coming from Koebe Theorem (Equation \eqref{eq:koebe}), we fix $\delta < \delta_0'$ and set for every $0 \leq j \leq n-1$ and $1 \leq i \leq q_j$:
 $$ \delta^j_i :=  \vert (h_i^j)'(t_{k_0}(p_0))  \vert \delta e^{-\epsilon (q_0 + \ldots + q_j + i)} . $$
 Note that
 \begin{equation}\label{telesco}
 \delta_{i+1}^j = \vert (h_{ t^j_{i}, \tilde k^j_{i}, \tilde k^j_{i+1} })' ( t^j_i) \vert   \delta^j_i    e^{-\epsilon} \ , \  \delta_{q_{n-1}}^{n-1} =  \vert h'_{\gamma_{[0,n]}}(t_{k_0}(p_0)) \vert \delta e^{-\epsilon \sum_{j=0}^{n-1} q_j } .
 \end{equation}
We also have $ \delta^j_i  \leq \vert (h_i^j)'(t_{k_0}(p_0))  \vert \leq e^{j(\lambda + (3+ \theta)\epsilon)}$ by Lemma \ref{needto}, precisely by the second inequality of assertion $(\tau_j)$ with $n$ replaced by $j$. In the sequel we shall increase $n_1$ given by Proposition \ref{l: lebesg} without specifying it. In particular, one can assume that
 $$\forall j \geq n_1 \ , \ \forall 1 \leq i \leq q_j \ , \ \delta^j_i < r_0 < \delta_0' ,$$
where $r_0$ comes from Proposition \ref{l: lebesg}. Moreover, by Propositions \ref{d: lyapunov exponent}, \ref{p: integral of beta} and Corollary \ref{c: estimq}, we have for every $n \geq n_1$:
\begin{enumerate}
\item[(a)] $\delta \vert h'_{\gamma_{[0,n]}}(t_{k_0}(p_0)) \vert \geq e^{n (\lambda-\epsilon)}$,
\item[(b)] $\int_{\gamma_{[0,n]}} \beta \leq n (-h_D + \epsilon)$ (see Equation \eqref{eq: Radon Nikodym derivative} for the link with $\beta$),
\item[(c)]  $\sum_{j=0}^{n-1} q_j \leq  n M_+$.
 \end{enumerate}  
We can also require for every $n \geq n_1$ (see Proposition \ref{l: lebesg} for (d)):
\begin{enumerate}
\item[(d)] $\sum_{j=0}^{n-1} \eta_{r_0} (\gamma^j) \leq n \epsilon$,
\item[(e)] $\nu_{k_0} (\D_{t_{k_0}(p_0)}(\delta)) \leq e^{n \epsilon}$.
 \end{enumerate} 
Now let us prove the upper estimate on $\nu_{k_n} \left( \D_{t_{k_n}  (p_n)} (  e^{n (\lambda - \epsilon(1+M_+)) }  ) \right)$. By (a) and (c), it is smaller than (see Equation \eqref{telesco} for the equality):
$$     \nu_{k_n} \left( \D_{t_{k_n}  (p_n)}\left (\vert h'_{\gamma_{[0,n]}}(t_{k_0}(p_0)) \vert \delta e^{-\epsilon \sum_{j=0}^{n-1} q_j } \right) \right) = \nu_{k_n} \left( \D_{t_{k_n}  (p_n)} ( \delta_{q_{n-1}}^{n-1} ) \right) ,  $$
which is  in turn smaller than $\nu_{k_0} (\D_{t_{k_0}(p_0)}(\delta))$ times the telescopic product 
$$ \prod_{j=0}^{n-1} \prod_{i=0}^{q_j-1} {  \nu_{\tilde k_{i+1}^j} \left(  \D_{t_{i+1}^j}  \left( \delta_{i+1}^j \right) \right)  \over   \nu_{\tilde k_i^j} \left(  \D_{t_i^j}   (   \delta^j_i ) \right)   }  . $$
Let us set $$S_n := \sum_{j=0}^{n-1} \sum_{i=0}^{q_j-1} \eta(t^j_i , \tilde k^j_i , \tilde k^j_{i+1} , \delta^j_i ) .$$ 
By the first inequality of Lemma \ref{eq:koebeRN} (with $\delta$ replaced by $\delta_i^j < \delta_0'$) and Equation \eqref{telesco}, the telescopic product is smaller than 
$$ e^ {S_n} \prod_{j=0}^{n-1} \prod_{i=0}^{q_j-1} {D(h_{t^j_i , \tilde k^j_i , \tilde k^j_{i+1}}^{-1})_* \nu_{\tilde k^j_{i+1}}  \over D\nu_{\tilde k^j_i}} (t^j_i) . $$
We now analyse separately $e^ {S_n}$ and the double product. Let us begin with the double product.
Let $\alpha_1$ be a leafwise path connecting $\gamma(0)$ to $\Sigma_{k_0}$ in the foliation box $U_{k_0}$, and let $\alpha_2$ be a leafwise path connecting $\gamma(n)$ to $\Sigma_{k_n}$ in the foliation box $U_{k_n}$,  whose $g_P$-length is smaller than $\log (\epsilon n)$ (if $k_n \in J$, hence near a singular point, we can use the path $\Lambda_{\gamma(n)}$ provided by Lemma \ref{lemma: angularproj}).
By Equation (\ref{eq: Radon Nikodym derivative}), the double product is equal to $\exp( \int_{\gamma^*} \beta )$, where $\gamma^*$ stands for $\gamma_{[0,n]}$ concatened with $\alpha_1$ and $\alpha_2$. 
By Lemma \ref{lemmaharnack}, the $1$-form $\beta$ is $g_P$-bounded. Using (b), we get $$ \int_{\gamma^*} \beta  \leq \int_{\alpha_1} \beta + \vert \beta \vert  \log(\epsilon n) +  n(-h_D + \epsilon)\leq  n(-h_D + 2\epsilon) , $$
where the last inequality holds for $n \geq n_0$ large.
Using $\nu_{k_0} (\D_{t_{k_0}(p_0)}(\delta)) \leq e^{n \epsilon}$ provided by (e), we obtain at this stage
$$ \forall n \geq n_0 \ , \ \nu_{k_n} \left( \D_{t_{k_n}  (p_n)}\left (  e^{n(\lambda - \epsilon(1+M_+))}  \right) \right)  \leq e^ {S_n}  e^{n(-h_D + 3\epsilon)} . $$ 
It remains to prove that $e^{S_n} \leq e^{2n \epsilon}$ for $n$ large enough. Since we chose $n_1$ such that $\delta^j_i \leq r_0$ for $j \geq n_1$ and $1 \leq i \leq q_j$, we obtain by definition of $\eta_{r_0}$:
$$ S_n = S_{n_1} + \sum_{j=n_1+1}^{n-1} \sum_{i=0}^{q_j-1} \eta(t^j_i , \tilde k^j_i , \tilde k^j_{i+1} , \delta^j_i ) \leq S_{n_1} + \sum_{j=n_1+1}^{n-1} \eta_{r_0}(\gamma ^ j) ,$$
which by (d) is smaller than $2n \epsilon$ for $n$ large enough, as desired. 
 
The second item of Proposition \ref{prop:expentropy} (lower estimate) is obtained similarly, by introducing counterparts of (a), (b), (c), (d), (e) and increasing $n_1(\gamma)$ if necessary. The main difference is that we require that $t_{k_0}(p_0)$ belongs to the support of $\nu_{k_0}$, in order to get the lower bound $\nu_{k_0} (\D_{t_{k_0}(p_0)}(\delta)) \geq e^{-n \epsilon}$ for (e). This completes the proof of Proposition \ref{prop:expentropy}.

\section{Dimension of transversal measures (proof of Theorem \ref{thmB})}\label{proofthmB}

We begin with the following Lemma.

\begin{lemma}\label{prop: double}
Let $\epsilon' >0$. There exists $\Lambda \subset (S^* \times K)^\N$ of full $W^1_K$-measure such that for every $w = ( p_n,k_n )_{n \in \N} \in \Lambda$ and for every $n \geq n_2(w)$, 
\begin{enumerate}
\item[1.] $\nu_{k_0} (\D_{t_{k_0}  (p_0)} ( e^{n(\lambda- \epsilon(1+M_+))}  )) \leq e^{-n (1-3 \epsilon' ) (h_D -5\epsilon)}$,
\item[2.] $\nu_{k_0} (\D_{t_{k_0}  (p_0)} ( e^{n (1-3 \epsilon' ) (\lambda + \epsilon(1+M_-)) } )) \geq e^{-n  (h_D  + 5 \epsilon)}$.
\end{enumerate}
\end{lemma}

\begin{proof} 
We use the natural extension introduced in Section \ref{ss: diext}. Let $\Pi : ( S^* \times K)^\Z \to ( S^* \times K)^\N$ be the mapping $\Pi ( (p_n , k_n)_{n \in \Z} ) := (p_n , k_n)_{n \in \N}$. Let $n_1$ be given by Proposition \ref{prop:expentropy} and let $N$ be large enough such that  
$$ \widehat \Omega := \{ \widehat w  \in (S^* \times K)^\Z  \ , \  n_1 \circ \Pi(\widehat w) \leq N \}  $$
satisfies $\widehat W_K (\widehat \Omega) \geq 1 - \epsilon'$. For every $\widehat w \in (S^* \times K)^\Z$, we denote 
$$\Sigma_n (\widehat w) :=\left  \{ k \in \{ 0 , \ldots , n \}  \ , \ \widehat \sigma_K^{-k} (\widehat w) \in \widehat \Omega \right  \} . $$
By Birkhoff ergodic Theorem, there exists $\widehat \Lambda \subset  (S^* \times K)^\Z$ of full $\widehat W_K$-measure such that for every $\widehat w \in \widehat \Lambda$, there exists $p(\widehat w) \in \N$ satisfying 
$$  \forall n \geq p(\widehat w) \ \ , \ \  \# \Sigma_n (\widehat w) \geq n (\widehat W_K (\widehat \Omega) - \epsilon') \geq n(1-  2\epsilon') . $$
Let $\Lambda := \Pi ( \widehat \Lambda)$. For every $w \in \Lambda$, let $\widehat w = (p_n , k_n)_{n \in \Z} \in \widehat \Lambda$ such that $\Pi(\widehat w)= w$. For every $n \geq n_2(w) := \max \{ p(\widehat w) , N / (1-3\epsilon') \}$, let 
$$m_n \in \Sigma_n (\widehat w) \cap \{ n(1-3\epsilon') , \ldots , n \} . $$
Since $m_n \in \Sigma_n (\widehat w)$, we have $\widehat \sigma_K ^{-m_n} (\widehat w) \in \widehat \Omega$, and thus we get 
\begin{equation}\label{enca}
 n_1 \circ \Pi (\widehat \sigma_K^{-m_n} (\widehat w)) \leq N \leq n (1-3\epsilon') \leq m_n \leq n .
 \end{equation}
Then, Proposition \ref{prop:expentropy} applied to $\Pi(\widehat \sigma_K^{-m_n} (\widehat w))$ and to the integer $m_n$ yields 
$$\nu_{k_0} ( \D_{t_{k_0}  (p_0)} ( e^{m_n (\lambda - \epsilon(1+M_+)) }  ) ) \leq e^{-m_n (h_D - 5 \epsilon) } . $$
By using Equation (\ref{enca}), we obtain finally (since $\lambda < 0$, $h_D > 0$ and $\epsilon \ll 1$)
$$ \nu_{k_0} (\D_{t_{k_0}  (p_0)} (e^{n (\lambda - \epsilon(1+M_+))} )  \leq  e^{-n(1-3\epsilon')(h_D -5 \epsilon)}   $$
 for every $n \geq n_2(\omega)$. The lower estimate is proved similarly, by applying the second estimate of Proposition \ref{prop:expentropy}.
 \end{proof}

To complete the proof of Theorem \ref{thmB}, it now suffices to show that for every $k \in K$, for $\nu_k$-almost every $t$ and for every $n \geq n_2(k,t)$, we have
\begin{enumerate}
\item[(i)] $\nu_k (\D_t ( e^{n(\lambda- \epsilon(1+M_+))}  )) \leq e^{-n (1-3 \epsilon' ) (h_D -5\epsilon)}$,
\item[(ii)] $ \nu_k (\D_t ( e^{n (1-3 \epsilon' ) (\lambda + \epsilon(1+M_-)) } )) \geq e^{-n  (h_D  + 5 \epsilon)}$.
\end{enumerate}

Recall that $(\pi_0)_* W^1_K= \mu$, where $\pi_0$ is the projection $(p_n , k_n)_n \mapsto p_0$, see Section \ref{ss: diext}. Let $k \in K$ and $t_k : U_k \to \D$ be the first integral defined in Section \ref{sss: cov}. In $U_k$,  $\mu = \int \int _{{\bf D} \times \{t\} } \tau_k (\cdot, t) d\text{vol}_{g_P}(\cdot)  \nu_k (dt)$  
where $\tau_k(\cdot, t)$ are positive harmonic functions. It follows that $(t_k)_* \mu$ is equivalent to $\nu_k$ by Harnack inequality \eqref{eq: harnack}. We deduce the aimed properties $(i)$ and $(ii)$ from $(\pi_0)_* W^1_K= \mu$ and from Lemma \ref{prop: double}, whose estimates hold $W^1_K$-generically.

\part{Entropies and discreteness}

\section{Leaf entropy $h_L$}\label{s: entropies} In the preceding sections we used the dynamical entropy $h_D$ (defined  in Section \ref{ss: DE}) to estimate the dimension of the current $T$. The \emph{leaf entropy} $h_L$ is another notion of entropy introduced by Kaimanovich \cite{Kaimanovich} and Ledrappier \cite{L3}. Endowing the leaves with the Poincar\'e metric \(g_P\), let \(p(x,y,\Timet)\) be the heat kernel associated to \(g_P\) on the leaf \( L_x\). 
We have 
\begin{equation} \label{eq: entropykingman}  h_L := \lim_{\Timet\rightarrow +\infty} - {1 \over \Timet} \int_{S^*} \int_{L_x}  p(x,y,\Timet) \log p(x,y,\Timet) \mathrm{vol}_{g_P} (dy) d\mu(x).
\end{equation} 
It satisfies (see  \cite[Section III]{L3}) for \(\mu\)-a.e. \(x\in S\):
\begin{equation} \label{eq: convergence entropy} h_L = \lim _{\Timet\rightarrow +\infty} - \frac{1}{\Timet} \int_{L_x}  p(x,y,\Timet) \log p(x,y,\Timet) \mathrm{vol}_{g_P} (dy) . \end{equation}
This number is well-defined since leaves are complete of bounded curvature, and appart the union of the separatrices that have vanishing \(\mu\)-measure, the injectivity radius of the leaves is bounded from below by a positive constant. 

Kaimanovich \cite[Theorem 4]{Kaimanovich} proved that $h_D \leq h_L$, we give a proof in Appendix for sake of completeness. We now prove Theorem \ref{t: dischh}, namely:\\

\emph{Let $\FF$ be a foliation on an algebraic surface $S$ with hyperbolic singularities and no foliated cycle.  If the holonomy pseudogroup of $\FF$ in restriction to the pseudo-minimal set is discrete, then $h_D = h_L$.}\\

{\bf Question:} Let $\FF$ be a foliation on an algebraic surface $S$ with hyperbolic singularities and no foliated cycle. If $h_D = h_L$, is the holonomy pseudogroup of $\FF$ discrete in restriction to the pseudo-minimal set? 

\section{$h_D = h_L$ (proof of Theorem \ref{t: dischh})}

It suffices to prove that $h_L \leq h_D$ when the holonomy pseudogroup of $\FF$ is discrete in restriction to the pseudo-minimal set. This follows from Proposition \ref{p: separated} below, which extends \cite[Proposition 8.2]{Deroin-Dupont} to singular foliations (in particular Item (3)). We shall need new features to neutralize the singularities of $\FF$, in particular the \(\mu\)-integrability of the function $\rho$.   \\

We fix a $\mu$-generic point\footnote{$B$ is the neighbourhood of the singular set, see Section \ref{ss: discretization}.} $x \in S \setminus B$.  We use the covering of $S^*$ by foliation boxes $(U_k,t_k)_{k \in K}$ constructed  in Section \ref{sss: cov}.

- We fix $k_x \in I \setminus J \subset K$ such that $x \in U_{k_x} \simeq \D \times \D$, 

- $\Sigma_x := \Sigma_{k_x}$ denotes the transversal\footnote{$x$ may not belong to $\Sigma_x$.} $\{ 0 \} \times \D \subset U_{k_x}$,

- $t_x := t_{k_x} : U_{k_x} \to \D$ denotes the local first integral, 

- with no loss of generality, $t_x(x) =0 \in \D$,

- $\nu_x$ denotes the pushforward measure $(t_x)_* (T_{\vert \Sigma_x}$). 

Let $\D_w(r) \subset \D$ denote the disc centered at $w$ of radius $r$. By Koebe distortion Theorem, there exists $\kappa> 0$ such that for every holomorphic injective function $h : \D_0(2 \rho) \to \D$, 
\begin{equation}\label{KDT}
 \forall r \leq \rho  \  , \ \D_{h(0)} (e^{-\kappa} |h'(0)| r ) \subset h (\D_0 (r) ) \subset \D_{h(0)} (e^{\kappa} |h'(0)| r ) .
 \end{equation}
The Lyapunov exponent $\lambda$ and the dynamical entropy $h_D$ have been introduced in Sections \ref{ss: LE} and \ref{ss: DE}.
By Lemma \ref{lemmaharnack}, the $1$-form $\beta$ is bounded with respect to the Poincar\'e metric $g_P$, let $\vert \beta \vert$ denote an upper bound.

\begin{proposition}\label{p: separated} 
Let $\FF$ be a foliation on an algebraic surface $S$ with hyperbolic singularities and no foliated cycle. Assume that the holonomy pseudogroup of $\FF$ is discrete in restriction to the pseudo-minimal set. Let $\epsilon > 0$. For every $n \geq N(\epsilon)$, there exist
\begin{enumerate}
\item[(i)] a set $I_n' \subset {\LL}_x \cap \Sigma_x$ of cardinality larger than $e^{n(h_L-3\epsilon)}$,
\item[(ii)]  leafwise continuous paths $$\gamma_z : [0,n+1] \to \LL_x$$ with endpoints $x$ and $z \in I_n' \subset L_x \cap \Sigma_x$,
\end{enumerate}
such that for some $r_1>0$ and for every $z\in I_n'$,  the holonomy map 
$$h_{x,z} := h_{\gamma_z , t_x, t_x}$$ satisfies  the following properties
\begin{enumerate}
\item[(1)] $h_{x,z} (\D_0(r_1)) \subset \D_0(r_1)$,
\item[(2)] the sets $h_{x,z}(\D_0(r_1))$ for $z \in I_n'$ are pairwise disjoint,
\item[(3)] $\nu_x (h_{x,z} (\D_0(r_1))) \geq \nu_x ( \D_{t_x(z)} ( e^{n (\lambda + \epsilon(1+M_-)) }  )) \geq e^{-n  (h_D  +  5 \epsilon + \vert \beta \vert \epsilon)}$.
\end{enumerate}
\end{proposition}

Using properties (1), (2), (3) and then (i), we infer for every $n \geq N(\epsilon)$,
$$\nu_x ( \D_0 (r_1) )  \geq  \sum_{z \in I_n'} \nu_x (h_{x,z}(\D_0(r_1))) \geq e^{n(h_L-3\epsilon)} e^{-n  (h_D  + 5 \epsilon + \vert \beta \vert \epsilon)} . $$
This implies $h_L - 3 \epsilon \leq h_D + 5 \epsilon + \vert \beta \vert \epsilon$. Letting $\epsilon$ to zero, we obtain $h_L \leq h_D$. Now we prove  Proposition \ref{p: separated}.

\subsubsection{Discretization of entropy}\label{ss: dientropy}
Here we interpret $h_L$ as the exponential growth of separated positions for leafwise Brownian paths  starting at $\mu$-generic points. Let $d_P$ be the distance on $\LL_x$ induced by the leafwise Poincar\'e metric. A subset $Z \subset \LL_x$ is a \emph{$(C,D)$-lattice} if its elements are $C$-separated and if any element of $\LL_x$ lies at a distance $\leq D$ from $Z$ for the distance $d_P$. Shortly, $Z$ is a \emph{lattice} if it is a $(C,D)$-lattice for some $C,D$. Let  
$$ \forall z \in \LL_x \ , \  V(z) := \cap_{z' \in Z} \  \{ y \in \LL_x \, , \,  d_P(y,z) \leq d_P(y,z')   \} , $$ 
this is the tile of the Voronoi tesselation associated to $Z$ containing $z$. For $n \geq 1$ let $B_n$ be the distribution at  time $n$ of leafwise Brownian motion starting at $x$ (it has density $p(x,y,n) \mathrm{vol}_{g_P} (dy)$). Let $v$ be the projection $\LL_x \rightarrow Z$ which contracts every tile to its center, it is defined $B_n$-a.e.. Let $\beta_n := v_* B_n$, this is a probability measure on $Z$. Proposition \ref{p: adapt} below extends the following result to separated subsets which are almost lattices.

\begin{proposition} {\cite[Theorem 3]{K2}, \cite[p. 195]{L3}} \label{l: signification of entropy}
Let $Z$ be a lattice in $\LL_x$ and $(Z_n)_n$ be subsets of $Z$ such that $\beta_n (Z_n) \geq 1/2$. Then $\liminf_{n\rightarrow \infty} \frac{1}{n} \log \  \lvert Z_n \rvert \geq h_L$.
\end{proposition} 

We recall that $W^x_{g_P}$ is the Wiener measure on the set $\Gamma^x$. 

\begin{proposition}\label{p: adapt}
Let \( Z' \subset \LL_x \) be a $d_P$-separated subset. Assume that
 $$ \forall \gamma  \in \Gamma^x \ ,  \ W^x_{g_P}-a.e \ , \ d_P (\gamma (\Timet) , Z' ) = o (\Timet) . $$ 
Let \( E \subset \Gamma^x\) be such that $W^x_{g_P}(E) > 1/2$ and let \( E_n := \{\gamma (n), \gamma\in E\}\subset \LL_x \). Then 
\[ \liminf _{n \rightarrow +\infty} \frac{1}{n} \log \ \lvert  v' ( E_n ) \lvert \geq h_L , \] 
where \( v'  : \LL_x \rightarrow Z' \) sends $B_n$-almost every $y \in \LL_x$ to its nearest point  in $Z'$.
\end{proposition}

\begin{proof} 
By assumption, for $W^x_{g_P}$-a.e. $\gamma \in \Gamma^x$, there exists \(  n_0 (\gamma) \) such that
\[ d_P(\gamma(n) , Z') \leq \varepsilon n \text{ for } n\geq n_0(\gamma)  .\]
Let \(n_0\) be large enough so that  $F := \{ \gamma \in E \, , \, n_0 (\gamma) \leq n_0 \}$ has $W^x_{g_P}$-measure $> 1/2$. Let \( F_n := \{\gamma (n), \gamma\in F \}\subset E_n \) and let $Z \subset \LL_x$  be a lattice containing \(Z' \).  By Proposition \ref{l: signification of entropy}, 
\begin{equation} \label{eq: discretization entropy}  \liminf _{n \rightarrow +\infty} \frac{1}{n} \log \, \vert v(F_n) \vert \geq h_L. \end{equation} 
The subset \( v(F_n) \subset Z \) is partitioned by the fibers of \( v' \). The diameter of every atom of that partition is \(\leq  2 \varepsilon n \), hence the number of points of $Z$ in every atom is \( \leq e^{2 c \varepsilon n} \), where \(c\) is a constant fixed by the hyperbolic area of unit balls in $\LL_x$. The number of atoms is also $\vert v'(F_n) \vert$. We infer 
\[ \vert v'(E_n) \vert \geq \vert v'(F_n) \vert \geq  \vert v(F_n) \vert e^{-2c\varepsilon n}  .  \]
We conclude by using Equation \eqref{eq: discretization entropy} and taking $\epsilon$ arbitrary small. 
\end{proof} 

\subsubsection{Construction of separated positions at time $n$}\label{ss: csp}
We apply Proposition \ref{p: adapt}. The functions  $\delta_\gamma$, $C_\gamma$, $n_1(\gamma)$  and $m_0(\gamma)$ are defined in Propositions \ref{prop:holodist}, \ref{prop:expentropy} and Lemma \ref{lemma: angularproj}. 
Let \[ \Gamma^{x,\epsilon} := \{ \gamma \in \Gamma^x \ , \ \delta_\gamma \geq 2 r_0 \, , \, C_\gamma \leq C_0 \, , \,  m_0 (\gamma) \leq N_0 \, , \, n_1(\gamma) \leq N_0 \}  , \]
where $r_0$ is small and $C_0$, $N_0$ are large such that $W^x (\Gamma^{x,\epsilon}) > 1/2$. In the remainder we may increase $N_0$ without specifying it.  By Proposition \ref{p: lyapunov product}, for every $\gamma \in \Gamma^{x,\epsilon}$ and $n \geq N_0$, we have 
\begin{equation}\label{lylyly}
 \lambda -\epsilon  \leq {1 \over n} \log \lvert h'_{\gamma_{[0,n]} , k_x, k_{\gamma(n)}}(x) \rvert_{m_{pr}} \leq  \lambda + \epsilon .
 \end{equation}
Lemma \ref{l: inthedomain} implies that if $\gamma(n) \in B$, then for every $n \geq N_0$:
\begin{equation}\label{eq: gsgp}  
 d_P(\gamma(n) , \partial B) \leq \log (\epsilon n) .
 \end{equation}
Let $r_1 := r_0 e^{-2\kappa}/24$, where $\kappa$ is defined in Equation \eqref{KDT}.  The subset 
$$Z' := \LL_x \cap (\{ 0 \} \times \D_0 (r_1 / 2)) \subset \LL_x \cap \Sigma_x$$  is $C$-separated in $\LL_x$ for some $C>0$. Since $S \setminus B$ is compact, there exists $H \geq 1$ which does not depend on $\epsilon$ such that, by Equation \eqref{eq: gsgp}:
$$  \forall \gamma \in \Gamma^{x,\epsilon} \ , \ \forall n \geq N_0 \ , \  d_P (\gamma(n) , Z')  <  \log H + \log (\epsilon n) . $$
One can assume that $\epsilon H$ is small with respect to $\lambda_{g_P}$. We set $$\Gamma^{x,\epsilon} (n) := \{  \gamma(n) \, , \, \gamma \in \Gamma^{x,\epsilon} \}. $$
 Recall that $B_n$ is the distribution at  time $n$ of leafwise Brownian motion starting at $x$, and that \( v' \) sends $B_n$-almost every $y \in \LL_x$ to its nearest point  in $Z'$. We set 
 \begin{equation}\label{eq:JN}
 J_n := v' (\Gamma^{x,\epsilon}(n) ) \subset Z' \subset  \LL_x \cap \Sigma_x . 
 \end{equation}
Applying Proposition \ref{p: adapt} with $E:=\Gamma^{x,\epsilon}$, we get $\lvert J_n \rvert \geq e^{n (h_L -\epsilon) }$ for $n \geq N_0$. 
 
 \subsubsection{Bounded holonomy}

\begin{lemma}\label{l: cl1} There exist $E>0$ and $r_0' >0$ satisfying the following properties. Let $z \in J_n$ and let $\gamma \in \Gamma^{x,\epsilon}$ be such that $y := \gamma(n)$ satisfies $z = v'(y)$. Let $k_y \in K$ such that $y \in U_{k_y}$ and let $\nu_y := (t_{k_y})_* T_{\vert \Sigma_{k_y}}$.

There exists a smooth path $\hat \gamma : [0,1] \to  \LL_x$ with endpoints $y,z$ such that for every $0 \leq \rho \leq r_0'$:
\begin{enumerate}
\item $h_{\hat \gamma , k_y, k_x}$  is defined on $\D_{t_{k_y}(y)} (r_0')$ and  $- E \leq \log \lvert h'_{\hat \gamma , k_y, k_x} \rvert_{m_{pr}} \leq E$.
\item $e^{-n\epsilon \vert \beta \vert} \nu_{y}(\D_{t_{k_y}(y)}(\rho)) \leq  \nu_{x} ( h_{\hat \gamma , k_y,k_x} (\D_{t_{k_y}(y)}(\rho))) \leq e^{n\epsilon \vert \beta \vert} \, \nu_y(\D_{t_{k_y}(y)}(\rho))$. 
\end{enumerate}
 \end{lemma}

 \begin{proof}
If $y = \gamma(n) \in S \setminus B$, the two Items follow from the compactness of $S \setminus B$. Let us go back to this situation when  $y \in B$. In that case,  $y \in \pi^{-1}(U_j)$ for some $j := k_y \in J$ (see Section \ref{sss: cov} for the definition of $U_j$, $J$ and $\pi$). By Remark \ref{rk: idhol} there exists a smooth path $\tilde \gamma$ drawn on $\pi^{-1}(U_j)$ with endpoints $y \in B$ and $\pi(y) \in \partial B$ whose corresponding holonomy map is the identity map $t_j(U^+_j) \to t_j(U^+_j)$ (consider for instance  the path $\Lambda_{\gamma(n)}$ of Section \ref{s: angdom}).  This proves the first Item. For the second Item, observe that Equation \eqref{eq: gsgp} comes from $d_P(\gamma(n) , \partial B) \leq l_{g_P}(\Lambda_{\gamma(n)}) \leq \log (\epsilon n)$, hence the $g_P$-length of $\Lambda_{\gamma(n)}$ is bounded by $\log (\epsilon n)$, which is $\leq \epsilon n$. To conclude we use Equation \eqref{eq: Radon Nikodym derivative}, by integrating $\beta$ on $\Lambda_{\gamma(n)}$.  
 \end{proof}

 \subsubsection{Proof of Proposition \ref{p: separated}}
 The set $J_n = v' (\Gamma^{x,\epsilon}(n) )$ was constructed  in Section \ref{ss: csp}, it satisfies $\lvert J_n \rvert \geq e^{n (h_L -\epsilon) }$ for $n \geq N_0$.
 
 \begin{lemma}\label{l: cl2} 
 For every $n \geq N_0$, there exist 
\begin{enumerate}
\item[(i')] a set $I_n \subset J_n \subset \LL_x \cap \Sigma_x$ of cardinality larger than $e^{n(h_L-2\epsilon)}$,
\item[(ii')] a  negative number $\lambda_n \in [n (\lambda- 2 \epsilon),n(\lambda + 2 \epsilon)]$,
\item[(iii')] leafwise continuous paths $\gamma_z : [0,n+1] \to \LL_x$ with endpoints $x$ and $z \in I_n \subset L_x \cap \Sigma_x$,
\end{enumerate}
such that the holonomy map $h_{x,z} := h_{\gamma_z , t_x, t_x}$ satisfies for every $z \in I_n$: 
\begin{enumerate}
\item[(a)]  $h_{x,z}$ is defined on $\D_0(r_0)$,
\item[(b)] $\lambda_n \ \leq \ \log \lvert h'_{x,z}  (0) \lvert_{m_{pr}} \ \leq \ \lambda_n + \log 2$,
\item[(c)] $\D_{t_k(z)} (r  e^{\lambda_n} e^{-\kappa})  \subset  h_{x,z} (\D_0(r )) \subset \D_{t_k(z)} (r  e^{\lambda_n+\log 2} e^\kappa)$ for $r \leq r_0/2$,
\item[(d)] $h_{x,z} (\D_0(r_1)) \subset \D_0(r_1)$.
\item[(e)] $\nu_x (h_{x,z} (\D_0(r_1)) \geq \nu_x ( \D_{t_k(z)} ( e^{n (\lambda + \epsilon(1+M_-)) }  )) \geq e^{-n  (h_D  + 5 \epsilon + \vert \beta \vert \epsilon)}$.
\end{enumerate}
 \end{lemma}
 
 \begin{proof}  Let $z \in J_n$ and let $\gamma \in \Gamma^{x,\epsilon}$ be such that $y := \gamma(n)$ satisfies $z = v'(y)$. Let $k_y \in K$ satisfying $y \in U_{k_y}$ and $\psi_{k_y}(\pi(\gamma(n))) > 0$, in particular Remark \ref{rk: undiscr} applies. Let $\hat \gamma : [0,1] \to \LL_x$ be a smooth path with endpoints $y,z$ provided by Lemma \ref{l: cl1}. We set 
$$\gamma_z :=  \hat \gamma * \gamma_{ [0,n]} \ , \ h_{x,y} := h_{\gamma_{ [0,n]} , k_x, k_y} \ , \ h_{y,z} := h_{\hat \gamma,k_y, k_x} \ , \  h_{x,z} := h_{y,z} \circ h_{x,y} .$$ 
By definition of $\Gamma^{x,\epsilon}$, $h_{x,y}$ is well-defined on $\D_0(2r_0)$. By  Equations \eqref{KDT} (with $\rho=r_0$) and \eqref{lylyly},   the image of $\D_0(r_0)$ by $h_{x,y}$ is  contained in a disc of radius $r_0 e^{n (\lambda + \epsilon)} e^\kappa$. One can assume that $r_0 e^{n (\lambda + \epsilon)} e^\kappa \leq r_0'$ ($r_0'$ is provided by Lemma \ref{l: cl1}). Then $h_{x,z} = h_{y,z} \circ h_{x,y}$ is well-defined on $\D_0(r_0)$, it implies $(a)$ for every $z \in J_n$. Moreover by Equation \eqref{lylyly} and Lemma \ref{l: cl1}: 
\[ \forall z \in J_n \ , \  - E  + n (\lambda -\epsilon )  \leq \log \lvert h_{x,z} '(0) \rvert_{m_{pr}} \leq  E  + n (\lambda + \epsilon ) .\]
This ensures the existence of $\lambda_n$ satisfying 
\begin{equation}\label{EEE}
 - E + n (\lambda -\epsilon )  \leq \lambda_n \leq   E + n (\lambda + \epsilon )
 \end{equation}
 and of a subset $I_n\subset J_n$ of cardinality 
\begin{equation*}\label{ent} 
| I_n| \geq \frac{|J_n|\log 2} { 2 (E + \epsilon n)}
\end{equation*}
such that $(b)$ is satisfied for every $z \in I_n$. Items (i') and (ii') then respectively follow from $\lvert J_n \rvert \geq e^{n (h_L -\epsilon) }$ and from Equation \eqref{EEE}, up to increasing $N_0$. Equation (\ref{KDT})  and $(b)$ ensure $(c)$. Recall that $r_1 = r_0 e^{-2\kappa} / 24$. We get $(d)$  from $h_{x,z}(0) = t_k(z) \in \D_0(r_1/2)$ and $\text{diam} \, h_{x,z} (\D_0(r_1)) \leq 2 r_1 e^{\lambda_n+\log 2}e^\kappa$ provided by $(c)$. 
Finally, $(e)$ follows from Proposition \ref{prop:expentropy} (precisely Remark \ref{rk: undiscr} applied to $\gamma_{ [0,n]}$) and the second Item of Lemma \ref{l: cl1} applied to $\hat \gamma$. 
\end{proof}

Now we infer all the items (i), (ii), (1), (2) and (3) of Proposition \ref{p: separated}. We build below $I_n' \subset I_n$ satisfying (i) and $(2)$. Since $I_n' \subset I_n$, items (ii), (1) and (3) directly follow from Lemma \ref{l: cl2}. We note that the construction of $I_n'$ satisfying (i) and $(3)$ was already performed in \cite{Deroin-Dupont}, we reproduce the argument for  completeness. Let $\kappa' := 2\kappa + \log 2$ and   
$$\mathcal H := \{ h: \D_0(r_1) \rightarrow \D \textrm{ holonomy map } , \,    - \kappa ' \leq \log |h'(0)| \leq  \kappa ' \} . $$ 
This set is finite because the holonomy pseudogroup of $\FF$ is discrete. Let $\rho_n := 2 r_1 e^{\lambda_n + \log 2} e^\kappa$ and $(z,z') \in I_n \times I_n$ satisfying  $|z - z'| \leq 2 \rho_n$. Item (c) yields (recall that $r_1 = r_0 e^{-2\kappa} / 24$, so that $3 \rho_n = r_0 e^{\lambda_n} e^{-\kappa} / 2$) 
\[ h_{x,z} \, (\D_0( 2 r_1 ) )\subset  \D_{t_x(z)} (\rho_n) ,\]
\[ h_{x,z'} \, (\D_0( r_0 / 2 ) )\supset  \D_{t_x(z')} (r_0 e^{\lambda_n} e^{-\kappa} / 2) = \D_{t_x(z')} ( 3 \rho_n)  \supset \D_{t_x(z)} (\rho_n). \]
We deduce that $h := h_{x,z'} ^{-1} \circ h_{x,z} : \D_0(2r_1) \to \D_0(r_0/2)$ is well defined and that  $- \kappa ' \leq \log |h'(0)| \leq  \kappa '$ by using $(a)$, $(b)$ and Koebe distortion. Hence the restriction of $h$ to $\D_0(r_1)$ belongs to  $\mathcal H$. Let $I_n' \subset I_n$ be such that $\{ h_{x,z} \}_{z \in I_n'}$ represents the classes of $I_n$ modulo $\mathcal H$ by right composition. We have 
$$|I_n'| \geq |I_n| / \vert \mathcal H \vert  \geq e^{n (h_L-2\epsilon)} / \vert \mathcal H \vert \geq e^{n (h_L-3\epsilon)} ,  $$ 
this proves (i). Item (3), that is $\{ h_{x,z}(\D_0(r_1)) \}_{z \in I'_n}$ are pairwise disjoint, follows from $\text{diam} \, h_{x,z} (\D_0(r_1)) \leq \rho_n$ and  $|z - z'| > 2\rho_n$ for every $(z,z') \in I'_n \times I'_n$. This completes the proof of Proposition \ref{p: separated}.

\section{Jouanolou foliation (proof of Corollary \ref{c: dimension Jouanolou})}

The Jouanolou foliation \(\mathcal J_2\) of degree \(2\) is the holomorphic foliation of $\mathbb P^2$ defined in homogeneous coordinates by  $y^2 \partial _ x + z^2 \partial _ y + x^2 \partial _z$. Let \(\mathcal J \) denote a holomorphic foliation on $\mathbb P^2$ topologically conjugate to \(\mathcal J_2\). Since it has degree $2$, the Lyapunov exponent $\lambda$ of the normalized harmonic current $T$ of $\mathcal J$ is equal to $- {d+2 \over d-1}$ with $d=2$, hence $\lambda = -4$.  \\

{\bf 1- \(\mathcal J \) has hyperbolic singularities and no foliated cycle.} 
It suffices to  verify that \(\mathcal J_2\) has hyperbolic singularities, this property being invariant by topological conjugacy \cite{CLNS84}. The eigenvalues at the $7$ singular points  are \( \lambda_{\pm} = -2\pm i\sqrt{3}\)  \cite{Jouanolou, LinsNetoJouanolou}: their quotient\footnote{The multipliers of holonomy maps of \(\mathcal J_2\) around the separatrices are therefore  \( \exp (2i\pi \lambda_\pm / \lambda_ \mp )\). They could become real in the stability component of \(\mathcal J_2\).}  does not belong to $\R$. 

Now let us verify that \(\mathcal J_2\) does not carry any foliated cycle, this property being also  invariant by topological conjugacy. The support of a foliated cycle must intersect the singular set  \cite[Theorem 2]{CLNS88}. Since the singular points of \(\mathcal J_2\) are hyperbolic, close to such a point, a foliated cycle would be supported on the union of the two separatrices, producing an algebraic leaf. The transcendence of the leaves \cite{Jouanolou} therefore prohibits the existence of a foliation cycle. \\

{\bf 2 - The holonomy pseudo-group of $\mathcal J$ is discrete} in restriction to the support of its pseudo-minimal set. That property is invariant by topological conjugacy, let us prove it for \(\mathcal J_2\).
Let \(\mathcal M_2\) be the pseudo-minimal set of \(\mathcal J_2\). 

The restriction of \(\mathcal J_2\) to its Fatou set  \( F(\mathcal J_2) \) is a fibration by discs over the Klein quartic \cite{AlvarezDeroin}. In particular, the Julia set \(J(\mathcal J_2)\) is a strict subset of \(\mathbb P^2\). Since \(\mathcal M_2 \subset J (\mathcal J_2)\)  by minimality,  \(\mathcal M_2 \) is also a strict subset of \(\mathbb P^2\).

We know that \(\mathcal J_2\) does not have any exceptional minimal set \cite{Camacho de Figueiredo}. In particular, \(\mathcal M_2\) contains a singular point of \(\mathcal J_2\). Since this one is hyperbolic, \(\mathcal M_2\) must contain at least one separatrix. Since the multiplier of the holonomy around this separatrix is not a real number, the regular part  \(\mathcal M_2' \) of \(\mathcal M_2\) is not a real analytic subset of \(\mathbb P^2 \setminus \text{sing} (\mathcal J_2)\). 

Assume by contradiction that the holonomy pseudo-group of \(\mathcal J_2\) in restriction to \(\mathcal M_2\) is not discrete. Since the holonomy pseudo-group of \(\mathcal M_2 '\) contains a hyperbolic map (associated to the separatrix above), and that it acts minimally on \(\mathcal M_2' \), one deduces by renormalization \cite[Proposition 2.0]{LorayRebelo} that the closure of the holonomy pseudo-group in restriction to \(\mathcal M_2\) contains a non trivial one parameter flow generated by a holomorphic vector field. Conjugating this flow with  a holonomy map having a non real multiplier and a fixed point in \(\mathcal M_2\), one finds another flow which is \(\mathbb R\)-independant from the first one at a generic point of \(\mathcal M_2\). In particular, \(\mathcal M_2\) has non empty interior. By minimality this implies \( \mathcal M_2 = \mathbb P^2\), a contradiction. Applying Theorem \ref{t: dischh}, one deduces that $h_D (\mathcal J)= h_L (\mathcal J)$ and that the transverse Hausdorff dimension of the harmonic current of  \(\mathcal J \) is equal to $h_L (\mathcal J)/ 4$. It remains to verify  $h_L (\mathcal J) = 1$ to complete the proof of Corollary \ref{c: dimension Jouanolou}.\\

{\bf 3 - Computation of $h_L (\mathcal J)$.}  It was proved in \cite{AlvarezDeroin} that \(\mathcal J_2\) satisfies the Anosov property: its leaves are simply connected, except for a countable number of them (which are annuli). This property is also satisfied for \(\mathcal J\). The leaf $L_x$ is therefore simply connected (isomorphic to $\D$) for $\mu$-almost every $x \in S$, hence $h_L (\mathcal J)=1$ as desired. 
 
\section{Appendix: proof of Kaimanovich's inequality $h_D \leq h_L$}\label{appendix}
 
We saw in Section \ref{ss: DE} that the desintegration of the harmonic current \(T\) along  \(\mu \)-a.e. leaf yields a single-valued positive harmonic function \( \tilde \tau \), well defined up to multiplication by a positive constant.
Let us consider the continuous time Markov process on the leaf \(L_x\) with transition probabilities  
\begin{equation} \label{eq: transition probabilities} p(x,y,\Timet) \frac{{\tilde \tau}(y)}{{\tilde \tau}(x)} \text{vol} _{g_P} (dy).\end{equation}
These probabilities define a semi-group of operators of \(  L^{\infty} (\text{vol} _{g_P}) \):
\[ P ^\Timet (\chi) (x) = \int_{L_x} \chi(y) p(x,y,\Timet) \frac{{\tilde \tau}(y)}{{\tilde \tau}(x)} \text{vol} _{g_P} (dy) .\]
It satisfies
\[  \int_{L_x}  P^\Timet(\chi)(x)  {\tilde \tau} (x) \text{vol}_{g_P} (dx) = \int_x \int_y \chi(y) p(x,y,\Timet) {\tilde \tau} (y)   \text{vol}_{g_P} (dx) \text{vol}_{g_P} (dy) , \]
which is equal to $\int \chi(y) {\tilde \tau} (y)\text{vol}_{g_P} (dy)$, 
so the measure \({\tilde \tau} (x) \text{vol}_{g_P} (dx) \) is stationary for the process. As a consequence,  \(\mu\) is a stationary measure for the process defined by \eqref{eq: transition probabilities} on \(S\). But \(\mu\) is ergodic in the sense that each \(\mathcal F\)-saturated Borel subset has \(\mu\)-measure \(0\) or \(1\). 
Hence, the random ergodic theorem shows that for \(\mu\)-a.e. \(x\in S\), a.e. Brownian path starting at \(\gamma(0)=x\) equidistributes itself wrt \( \mu\).
As in Section \ref{s: entropies} (see Equation \eqref{eq: convergence entropy}), the following limit exists for \(\mu\)-a.e. \(x \in S \):  
\begin{equation}\label{eq: hR}
 h_R := \lim _{\Timet\rightarrow +\infty} - \frac{1}{\Timet} \int_{L_x} p(x,y,\Timet) \frac{{\tilde \tau}(y) } {{\tilde \tau}(x) } \log \left( p(x,y,\Timet) \frac{{\tilde \tau}(y) } {{\tilde \tau}(x) }
\right) \text{vol}_{g_P} (dy) ,
\end{equation}
and we also have
\begin{equation}\label{eq: hRkingman}
 h_R =  \lim _{\Timet\rightarrow +\infty}  \frac{1}{\Timet} \int _S \left( \int_{L_y}  p(x,y,\Timet) \frac{{\tilde \tau} (x)}{{\tilde \tau} (y)}  \log \left( p(x,y,\Timet)^{-1} \frac{{\tilde \tau}(y)}{{\tilde \tau}(x)}\right)    \text{vol} _{g_P} (dx) \right) d\mu (y) 
 \end{equation}

The number given by \eqref{eq: hR} is non negative. Indeed, Harnack inequality \cite{Moser} for heat kernel and for positive harmonic functions shows that there exists a constant \( C >0\) so that for any \(\Timet\geq 1\),  if \( d_P(y,y') \leq 1\) then 
\begin{equation}\label{eq: Harnack inequality} p(x,y,\Timet) \leq e^C p(x,y',\Timet) \text{ and } {\tilde \tau}(y) \leq e^C {\tilde \tau} (y') .\end{equation} 
In particular, given any \(y\in L_x\) we have 
\[ {\tilde \tau} (x) = \int_{L_x}  p(x,y', \Timet) {\tilde \tau}( y')  \text{vol}_{g_P}(dy') \geq \int _{B(y,1)} p(x,y', \Timet) {\tilde \tau}( y') \text{vol}_{g_P}(dy')   \]
\[ \geq e^{-2C} p(x,y,\Timet) {\tilde \tau}(y) \text{vol}_{g_P}(B(y,1)) . \]
Since the $g_P$-volume of leafwise balls of radius one is bounded from below by a positive constant, this proves that for \(\Timet\geq 1\), the densities \( p(x,y,\Timet) \frac{{\tilde \tau} (y)}{{\tilde \tau} (x)}\) are bounded from above by a constant independant of \(x,y,\Timet\). As a consequence,  
\[\int_{L_x} p(x,y,\Timet) \frac{{\tilde \tau}(y) } {{\tilde \tau}(x) } \log \left( p(x,y,\Timet) \frac{{\tilde \tau}(y) } {{\tilde \tau}(x) }
\right) \text{vol}_{g_P} (dy)\]
is bounded from above by a constant independant of \(x,\Timet\). We deduce by Equation \eqref{eq: hR} that $h_R \geq 0$, as claimed. 

Now let us prove that $h_L - h_D = h_R$.
Given a time \(\Timet>0\), denote by \( \mathcal E : \Gamma \to S \times S \) the map \(\mathcal E (\gamma) = (\gamma(0), \gamma(\Timet))\) and let  \(\nu := \mathcal E_* W^\mu_{g_P}\). Since \(\mu \) is stationary, the projections \( {\text{pr}_i}_* \nu\) on each factor are equal to \( \mu\), and \(\nu\) gives full mass to the graph of the foliation \(G(\mathcal F)  := \{ (x,y)\in S\times S, y \in L_x \} \).  By definition of \( W^\mu_{g_P}\), and the fact that \(\mu\) decomposes as an integral of harmonic volumes along the leaves, the conditional measures of \(\nu\) wrt the projection \(\text{pr}_1\) (resp. \(\text{pr} _2\)) are equal to the family of measures \(\nu_1^x \)  (resp. \(\nu_2^y\)) supported on \( L_x \) (resp. \( L_y \)) defined by  
\[\nu _1^x = p (x,y, \Timet) \text{vol}_{g_P} (dy)  \ , \ \text{ resp. }  \nu _2^y = p(x,y,\Timet) \frac{{\tilde \tau} (x)}{{\tilde \tau} (y)} \text{vol} _{g_P} (dx).\] 
Let us consider the function \( f_\Timet : G(\mathcal F) \rightarrow \mathbb R\) defined by 
\[ f_\Timet(x,y) = \log \left( p(x,y,\Timet)^{-1} \frac{{\tilde \tau}(y)}{{\tilde \tau}(x)}\right).\]
and let us compute $\int_{G(\mathcal F)} f_\Timet \, d\nu$ by two different ways.
On one hand, we have
\[ \int_{G(\mathcal F)} f_\Timet \, d\nu = \int _S \left( \int_{L_x} f_\Timet(x,y) \nu_1^x (dy) \right) d\mu(x) . \]
Using the expression of $f_\Timet$ and $\nu_1^x$, we get 
\[ \int_{G(\mathcal F)} f_\Timet \, d\nu =  \int _S h_\Timet(x)  \, d\mu(x) + \int_S \left( \int_{L_x} \log \left( \frac{{\tilde \tau}(y)}{{\tilde \tau}(x)} \right) p(x,y,\Timet) \text{vol}_{g_P}  (dy)   \right) d\mu (x)  \]
where \( h_\Timet(x) :=\int_{L_x}  \log \left(p(x,y,\Timet) ^{-1}\right) p (x,y, \Timet) \text{vol}_{g_P} (dy)\). Proposition \ref{p: cochd} gives 
\[ \forall \Timet > 0 \ , \  \int_{G(\mathcal F)} f_\Timet \, d\nu =  \int _S h_\Timet(x)  \, d\mu(x) - \Timet h_D . \]
Using Equation \eqref{eq: entropykingman}, we obtain
\[ \lim_{\Timet\rightarrow +\infty}  \frac{1}{\Timet} \int_{G(\mathcal F)} f_\Timet \, d\nu = h_L - h_D .\]
On the other hand, we have 
\[  \int_{G(\mathcal F)} f_\Timet \, d\nu  =  \int _S \left( \int_{L_y} f_\Timet(x,y) \,  \nu_2^y (dx) \right) \, d\mu(y) . \]
Using the expression of $f_\Timet$ and $\nu_2^y$, we get
\[  \int_{G(\mathcal F)} f_\Timet \, d\nu  = \int _S \left( \int_{L_y}   \log \left( p(x,y,\Timet)^{-1} \frac{{\tilde \tau}(y)}{{\tilde \tau}(x)}\right)   p(x,y,\Timet) \frac{{\tilde \tau} (x)}{{\tilde \tau} (y)} \text{vol} _{g_P} (dx) \right) \, d\mu (y) . \]
Using Equation \eqref{eq: hRkingman}, we obtain $\lim_{\Timet\rightarrow +\infty}  \frac{1}{\Timet}  \int_{G(\mathcal F)} f_\Timet \, d\nu = h_R$, which completes the proof of the formula $h_L - h_D = h_R$.

\end{document}